\documentclass{article}
\usepackage{bbm}
\usepackage{array}
\usepackage{latexsym}
\usepackage{amssymb, amsmath}
\usepackage{amsthm}
\usepackage{enumerate}
\usepackage{graphics}
\usepackage{mathtools}
\usepackage{mathrsfs}
\usepackage{bm}
\usepackage{comment}
\usepackage{a4wide}
\usepackage[title]{appendix}
\numberwithin{equation}{section}

\newtheorem{Definition}{Definition}[section]
\newtheorem{Theorem}[Definition]{Theorem}
\newtheorem*{Proof}{Proof}
\newtheorem*{proof_main}{Proof of Theorem 1.4}
\newtheorem{Lemma}[Definition]{Lemma}
\newtheorem{_corollary}[Definition]{Corollary}

\makeatletter
\newcommand{\opnorm}{\@ifstar\@opnorms\@opnorm}
\newcommand{\@opnorms}[1]{%
  \left|\mkern-1.5mu\left|\mkern-1.5mu\left|
   #1
  \right|\mkern-1.5mu\right|\mkern-1.5mu\right|
}
\newcommand{\@opnorm}[2][]{%
  \mathopen{#1|\mkern-1.5mu#1|\mkern-1.5mu#1|}
  #2
  \mathclose{#1|\mkern-1.5mu#1|\mkern-1.5mu#1|}
}
\makeatother
\DeclareMathOperator{\esssup}{ess\sup}
\DeclareMathOperator{\essinf}{ess\inf}
\newcommand{\1}{\mbox{1}\hspace{-0.25em}\mbox{l}}

\title{Exact Hausdorff dimension \\ of the spectral measure \\ for the graph Laplacian on a sparse tree}
\author{Kota Ujino}
\date{}

\begin{document}
\maketitle

\begin{abstract}
The Hausdorff dimension of spectral measure for the graph Laplacian is shown exactly 
in terms of an intermittency function. The intermittency function can be estimated by  using one-dimensional discrete Schr\"{o}dinger operator method.
\end{abstract} 

\section{Introduction and results}
\subsection{Introduction}
We study the graph Laplacian on a sparse tree and its Hausdorff dimension.
The Hausdorff dimension is defined for sets or measures. 
We estimate the Hausdorff dimension of the spectral measure for the graph Laplacian on a sparse tree.
Note that the Hausdorff dimension of a measure and that of the support of the measure are different from each other in general. 
If the spectra are purely point spectra, then the Hausdorff dimension of the spectral measure is zero. 
If the spectrum is purely absolutely continuous, then the Hausdorff dimension of the spectral measure is one. 
In this paper, we show the Hausdorff dimension of a sparse tree exactly. 

This paper is organized as follows: 
In the rest of Section $1$, we give the main result. 
In Section $2$, we give a decomposition of the graph Laplacian. 
From this, we can identify the graph Laplacian with one-dimensional discrete Schr\"{o}dinger operators with a sparse potential. 
In the Section $3$, we prove that the intermittency function gives the upper bound of the upper Hausdorff dimension. 
In Section $4$, we prepare to estimate the intermittency function. Here, we estimate the operator kernel, by using a quadratic form theory and Helffer-Sj\"{o}strand formula. 
In Section $5$, we estimate the intermittency function and prove the main theorem. 

We define a sparse tree. 
We say that $G=(V,E)$ is a graph, if $V$ is a countable set and $E\subset \{e \in 2^V \mid \#e=2 \}$. An element of $V$ $($resp. $E$$)$ is called the vertex $($resp. the  edge$)$.
Vertices $a,b\in V$ are said to be adjacent, if $\{a,b\} \in E$. We denote by $a\sim b$, if $a,b \in V$ are adjacent. 
Note that this definition implies that there are no egdes which are adjacent to itself.
Vertices $a,b \in V$ are said to be linked, if 
there exist 
$a_i \in V$, $i=1,2,...,n-1$ such that $a_i \notin \{a,b\}$, $a \sim a_1$, $a_i \sim a_{i+1}$, $i=1,2,...,n-2$, and $a_{n-1}\sim b$. 
Here $\{a, a_1,...,a_{n-1},b\}\in 2^V$ is called a path from $a$ to $b$.
Let the degree ${\rm deg}(a)$ of the vertex $a$ be defined by  
${\rm deg}(a)=\# \{b \in V \mid a\sim b \}$.
A graph is said to be locally finite, if ${\rm deg}(a)<\infty$ for any vertex $a \in V$.
A graph $G$ is said to be connected if any vertices $a,b \in V$ are linked.
We say that a graph $G=(V,E)$ is a tree, if  
$G$ is connected and for any vertices $a,b \in V$, there exists a unique path from $a$ to $b$. 
We fix a vertex $o$ of the tree $G$, and $o$ is called the root of $G$. 
A tree $G$ with a fixed root $o$ is called a rooted tree.
Let $p(a,b) \in 2^V$ be the unique path from the vertex $a$ to the vertex $b$. 
The metric $d(\cdot,\cdot)$ on $V$ is defined by
\begin{displaymath}
d(a,b)=
\begin{cases}
0,
&a=b,
\\
\#p(a,b)-1,
&a \neq b.
\end{cases}
\end{displaymath}
%Let $\mathbb{Z}_{\geq0}= \{0\} \cup \mathbb{N}$.
Let $G$ be a rooted tree with a root $o$, and let  
$S_n=\{a\in V \mid d(o,a)=n\}$ for $n=0,1,...$.
We say that a rooted tree $G=(V,E)$ is a spherically homogeneous tree if 
any vertices in $S_n$ have the same degree $d_n$.
A locally finite spherically homogeneous tree $G$ is uniquely determined by the sequence $(g_n)_{n=0}^{\infty}$,
\[
g_n=
\begin{cases}
d_0,& n=0,\\
d_n-1,& n\geq1.
\end{cases}
\]
\begin{Definition}\label{sparse}
Let $L_n=2^{n^n}$, $n=1,2,... ,$ and $\Gamma \in (0,1)$. 
We say that a locally finite spherically homogeneous tree $G=(V,E)$ is a $\Gamma$- sparse tree, if for any $n \geq 0$,
\begin{displaymath}
g_n=
\begin{cases}
[n^{\frac{1-\Gamma}{\Gamma}}],&n\in \{L_m \mid m\in \mathbb{N}\},\\
1,&n \notin \{L_m \mid m \in \mathbb{N}\}.
\end{cases}
\end{displaymath}
\end{Definition}

We define the graph Laplacian for the locally finite graph.
Let $G=(V,E)$ be a locally finite graph.
Let $l^2(V)$ be the set of square summable functions on $V$, and this is the Hilbert space with the inner product given by 
\begin{eqnarray}
\left(f,g\right)=
\sum_{u\in V}\overline{f(u)}g(u).\nonumber
\end{eqnarray}
Let $\mathcal{D}\subset l^2(V)$ be defined by
$\begin{displaystyle}
\mathcal{D}=
\left\{
f:V\rightarrow \mathbb{C}\mid \#{\rm supp}(f)<\infty
\right\}
\end{displaystyle}$.
Let $L$, $A$, and $D$ be operators with its domain $\mathcal{D}$ , and defined by
\begin{eqnarray}
Lf(u)&=&
\sum_{v\sim u}(f(v)-f(u)),\nonumber\\
Af(u)&=&
\sum_{v\sim u}f(v),\nonumber\\
Df(u)&=&
\sum_{v\sim u}f(u)={\rm deg}(u)f(u).\nonumber
\end{eqnarray}
These are called graph Laplacian, adjacency matrix, and degree matrix, respectively. The graph Laplacian $L$ is essentially self-adjoint, if the graph is connected \cite[Threom3.1]{Jorgensen}.
\begin{comment}
\begin{Lemma}\label{ess self-adjoint}
$L$ is essentially self-adjoint.
\end{Lemma}
\begin{proof}
\cite[Lemma 8]{self-adjoint}
\end{proof}
\begin{Proof}
\rm
Let $f\in\mathcal{D}(L^*)$ and $g\in\mathcal{D}$.
Since the support of $g$ is a finite set, 
\begin{eqnarray*}
( f, L g )&=&\sum_{u\in V}\overline{f(u)}\sum_{v:u \sim v}(g(v)-g(u))\\
&=&\sum_{u\in V}\sum_{v:u \sim v}\overline{f(u)}g(v)-\sum_{u\in V}\sum_{v:u \sim v}\overline{f(u)}g(u)\\
&=&\sum_{u\in V}\sum_{v:u \sim v}\overline{f(v)}g(u)-\sum_{u\in V}\sum_{v:u \sim v}\overline{f(u)}g(u)\\
&=&\sum_{u\in V}\overline{\sum_{u \sim v}(f(v)-f(u))}g(u).\\
\end{eqnarray*}
This implies that for any $f\in \mathcal{D}(L^*)$ 
\[
L^*f(u)=\sum_{v \in V;u\sim v}(f(v)-f(u)).
\]
Assume $f\in\rm{Ker}\it{(L^*\pm i)}$. Then
\begin{eqnarray*}
( f, L^* f ) &=&\sum_{u\in V}\overline{f(u)}\sum_{v;u\sim v}(f(v)-f(u))\\
&=&\sum_{(u,v)\in V\times V:u\sim v}(\overline{f(u)}f(v)-|f(u)|^2).
\end{eqnarray*}
This implies that $( f, L^* f )$ is a real number, but 
$( f, L^* f )=\mp i \|f\|^2$ by the assumption. 
Therefore we have $f=\bm{o}$, and $L$ is essentially self-adjoint, if $G$ is connected or 
\qed
\end{Proof}
\end{comment}

Let $(X,d_X)$ be a metric space, and $\mathcal{B}(X)$ be the Borel $\sigma$-field of $(X,d_X)$. Let $A$ be a subset of $X$ and the diameter $d_X(A)$ of $A$ be defined by $d_X(A)= \sup \{d_X(x,y) \mid x,y \in A\}$. 
Let $\delta>0$ and a family $\{A_i\}_{i=1}^{\infty}$ of subsets of $X$
is called a $\delta$-{\rm cover} of $A$, if  
$
A\subset \bigcup_{i=1}^{\infty}A_i
$ and 
$
\sup_{1\leq i < \infty}d_X(A_i)\leq \delta
$.
\begin{Definition}
Let $\alpha \in [0,\infty)$ and $\delta>0$. 
Let $h^{\alpha}_{\delta}, h^{\alpha}:2^X \rightarrow [0,\infty]$ 
%and $h^{\alpha}:2^X \rightarrow [0,\infty]$ 
be defined by
\begin{eqnarray}
h^{\alpha}_{\delta}(A)&=& \inf \left\{ \sum_{i=1}^{\infty}d_X(A_i)^{\alpha} \mid \{A_i\}_{i=1}^{\infty}\text{is a $\delta$-{\rm cover} of $A$}
\right\}, \nonumber \\
h^{\alpha}(A)&=&\lim_{\delta \rightarrow 0}h^{\alpha}_{\delta}(A). \nonumber
\end{eqnarray}
\end{Definition}
Here $h^{\alpha}$ is called $\alpha$-dimensional Hausdorff measure of $X$. 
Actually, the restriction of $h^{\alpha}$ to $\mathcal{B}(X)$ is a measure.
Let ${\rm dim}A$ be defined for $A\in 2^{X}$ by 
\[
{\rm dim}A= \sup \left \{\alpha \mid h^{\alpha}(A) \neq 0\right\}.
\] 
This is called the Hausdorff dimension of $A$. 
It follows that
if $0\leq \alpha < {\rm dim}A$, then $h^{\alpha}(A)=\infty$ and that 
if ${\rm dim}A <\alpha $, then $h^{\alpha}(A)=0$. 
Let $\mu : \mathcal{B}(X)\rightarrow [0,\infty]$ be a measure, and let the lower Hausdorff dimension  
${\rm dim}_*\mu$ and 
the upper Hausdorff dimension ${\rm dim}^*\mu$
 of $\mu$  
be defined by
\begin{eqnarray}
{\rm dim}_*\mu&=& \inf \left\{
{\rm dim}A \mid \text{$A\in \mathcal{B}(X)$ such that  }\mu(A) \neq 0
\right\}, \nonumber \\
{\rm dim}^*\mu &=& \inf \left\{
{\rm dim}A \mid \text{$A\in \mathcal{B}(X)$ such that } \mu(\mathbb{R}\setminus A)=0
\right\}.\nonumber
\end{eqnarray}
If $\alpha={\rm dim}_*\mu={\rm dim}^*\mu$, then $\mu$ is said to have the exact $\alpha$-Hausdorff dimension.

Let $L$ be the graph Laplacian of the $\Gamma$-sparse tree. The following lemma is proved by Breuer \cite{Breuer}.
\begin{Lemma}\label{former-result}
Let $H=-\overline{L}$, and 
let $E$ be the spectral measure of $H$ and 
$\tilde{E}$ be the restriction of $E$ to the interval $(0,4)$, 
where $\overline{L}$ is the clousre of $L$.
Then it follows that 
\begin{enumerate}[$(1)$]
\item
$\sigma_{\rm ac}(H)=\emptyset$,
$\sigma_{\rm pp}(H)\cap(0,4)=\emptyset$, 
$\sigma_{\rm sc}(H)\cap(0,4)=(0,4)$,
\item 
$\Gamma \leq {\rm dim}_*\tilde{E} \leq {\rm dim}^* \tilde{E} \leq \frac{2\Gamma}{1+\Gamma}$.
%is $(\Gamma-\epsilon)$-continuous and 
%$(\frac{2\Gamma}{1+\Gamma}+\epsilon)$-singular for any $\epsilon>0$.
\end{enumerate}
\end{Lemma}

We obtain the main theorem below.
\begin{Theorem}\label{main}
We suppose the same assumptions as Lemma \ref{former-result}.
Then
$\Gamma={\rm dim}_*\tilde{E}={\rm dim}^*\tilde{E}$, and 
$\tilde{E}$ has the exact $\Gamma$-Hausdorff dimension,
\end{Theorem}
This theorem implies the corollary below.
\begin{_corollary}
For any $\Gamma\in(0,1)$, the restriction of the spectral measure for the graph Laplacian on the $\Gamma$-sparse tree to the interval $(0,4)$ has the exact $\Gamma$-Hausdorff dimension.
\end{_corollary}

\subsection{Preceding results}
The spectral analysis of a sparse tree bears some similarities to the theory of one-dimensional discrete Schr\"{o}dinger operators with a sparse potential. 
In Simon-Stolz \cite{sparsepotential}, Schr\"{o}dinger operators with a sparse potential have singular continuous spectrum. 
Gilbert-Pearson \cite{Gilbert} finds a relationship between the behavior of subordinate solutions and the spectrum of one-dimensional Schr\"{o}dinger operators. 
Jitomirskaya and Last \cite{Jitomirskaya} show a relationship between the Hausdorff dimension of the spectral measure and the behavior of subordinate and non-subordinate solutions. Moreover, they estimate the Hausdorff dimension of the spectral measure by calculating the $L$-norm of non-subordinate solutions. This subordinate solution method is also used in \cite{Breuer}.
 
On the other hand, the relationship between the type of spectra and the time-averaged behavior of Schr\"{o}dinger operators is also studied.
RAGE theorem implies that if the initial state is singular continuous, then the time-averaged evolution goes to infinity. 
Barbaroux, Combes and Montcho \cite{Bar} give a lower bound of the time-averaged momentum of one-dimensional discrete Schr\"{o}dinger operators, by using the upper Hausdorff dimension. 
This also shows an inequality between the upper Hausdorff dimension and an intermittency function. 
It is, however, crucial to estimate the intermittency function exactly. 
Tcheremchantsev \cite{Tcheremchantsev} gives the intermittency function explicitly in the case of sparse potentials. 
We will apply \cite{Tcheremchantsev} to the graph Laplacian on a sparse tree.

\section{Preliminaries}
Threre are some decomposition methods for Schr\"{o}dinger operators on some trees. 
These methods stem from Naimark and Solomyak \cite{Naimark}. 
Breuer \cite{BreuerA}, Kostenko, and Nicolussi \cite{Kostenko} developped this method recently. 
They study the case of the continuum Kirchhoff Laplacian. 
Allard, Froese \cite{Allard} and Breuer \cite{Breuer} study the case of the graph Laplacian. 
We introduce their results as Lemma $\ref{jacobi identification}$. 
Their results imply that the graph Laplacian on the spherically homogeneous tree is identified with the direct sum of Jacobi matrices. 
Hence, it is sufficient to study Jacobi matrices instead of the graph Laplacian $H$.

Let $G=(V,E)$ be a spherically homogeneous tree determined by the sequence $\{ g_n\}_{n=0}^{\infty}$ and $H=-\overline{L}$, where
$\overline{L}$ is the closure of $L$. 
Let $\alpha_n=\#S_n$ for $n=0,1,...$ and $\alpha_{-1}=0$. 
Since   
$\{\alpha_n\}_{n=0}^{\infty}$ is non-decreasing, 
there exists a unique $N(k)\in \mathbb{N}\cup \{0\}$ such that 
$\alpha_{N(k)-1}< k \leq \alpha_{N(k)}$ for every $k \in \mathbb{N}$.
Let $k,n\in \mathbb{N}$ and let 
$d_{k}=(d_k(n))_{n=1}^{\infty}$ and $a_{k}=(a_{k}(n))_{n=1}^{\infty}$ be defined by the following: in the case of $k=1$,
\begin{eqnarray}
d_1(n)&=&
=
\begin{cases}
g_0&(n=1),\\
g_{n-1}+1&(n \geq 2),
\end{cases}
,\nonumber\\
a_{1}(n)&=&=-\sqrt{g_{n-1}}.\nonumber
\end{eqnarray}
In the case of $k\geq2$, 
\begin{eqnarray}
d_k(n)&=&g_{n+N(k)-1}+1,\nonumber\\
a_{k}(n)&=&-\sqrt{g_{n+N(k)-1}}.\nonumber
\end{eqnarray}
By calculatig straightfowardly, we see that for any $k,n \in \mathbb{N}$  
\begin{eqnarray}\label{4.1}
d_k(n)=a_k(n)^2+1-\delta_1(k)\delta_1(n),
\end{eqnarray}
where 
$\delta_j(k)=
\begin{cases}
1&(k=j),\\
0&(k \neq j).
\end{cases}
$
Let Jacobi matrices $H^{(k)},A^{(k)},D^{(k)} :l^2(\mathbb{N})\rightarrow l^2(\mathbb{N})$ be defined by 
\begin{displaymath}
H^{(k)}=
\renewcommand{\arraystretch}{1.6}
\left(
\begin{array}{ccccccc}
d_{k}({1})&-a_{k}({1})\\
-a_{k}({1})&d_{k}({2})&-a_{k}({2})\\
&-a_{k}({2})&d_{k}({3})&-a_{k}({3})\\
&&-a_{k}({3})&\ddots&\ddots\\
&&&\ddots\\
\end{array}\right),
\renewcommand{\arraystretch}{1}
\end{displaymath}
\begin{displaymath}
A^{(k)}=
\renewcommand{\arraystretch}{1.6}
\left(
\begin{array}{ccccccc}
0&a_{k}({1})\\
a_{k}({1})&0&a_{k}({2})\\
&a_{k}({2})&0&a_{k}({3})\\
&&a_{k}({3})&\ddots&\ddots\\
&&&\ddots\\
\end{array}\right)
\renewcommand{\arraystretch}{1},
D^{(k)}=
\renewcommand{\arraystretch}{1.6}
\left(
\begin{array}{ccccccc}
d_{k}({1})&\\
&d_{k}({2})&\\
&&d_{k}({3})&\\
&&&\ddots&\\
&&&\\
\end{array}\right).
\renewcommand{\arraystretch}{1}
\end{displaymath}
Note that $H^{(k)}=D^{(k)}-A^{(k)}$. The next lemma shows the decomposition of graph Laplacian.
\begin{Lemma}\label{jacobi identification}
$H$ and $\begin{displaystyle}
\bigoplus_{k=1}^{\infty}H^{(k)} 
\end{displaystyle}$
are unitarly equivalent.
\end{Lemma}
\begin{Proof}\rm
See Appendix \ref{Decomposition of the graph Laplacian}.
\qed
\end{Proof}

\section{Intermittency function and Hausdorff dimension}
In this section, we introduce an intermittency function and give an important inequality in Lemma $\ref{H}$ which shows that the intermittency function is the upper bound of Hausdorff dimension.

\begin{comment}
Let $\mu: \mathcal{B}^1\rightarrow [0,\infty]$ be a measure, and 
$\alpha \in (0,1)$. 
Let $D^{\alpha}\mu:\mathbb{R}\rightarrow [0,\infty]$ be defined by
\[
D^{\alpha}\mu(E)= \limsup_{\delta \rightarrow 0}\cfrac{\mu((E-\delta, E+\delta))}{\delta^{\alpha}}.
\]
Let $\alpha \in [0,\infty)$. 
The measure $\mu$ is said to be $\alpha$-continuous, if 
$\mu(A)=0$ for any $A\in \mathcal{B}(X)$ with $h^{\alpha}(A)=0$, 
$\mu$ is said to be $\alpha$-{\rm singular}, if 
there exists $A\in \mathcal{B}(X)$ such that $\mu(X\setminus A)=0$ and 
$h^{\alpha}(A)=0$, and $\mu$ is said to be $\alpha$-absolutely continuous, if 
there exists a measureble function $f$ such that $d\mu=fdh^{\alpha}$.
\end{comment}

Let $\psi \in l^2(\mathbb{N})$ and  
$E^{(k)}$ be the spectral measure of $H^{(k)}$.
We consider the time-averaged dynamics of $\text{exp}(-itH^{(k)})\psi$. 
Let a finite measure $\mu_{\psi}^{(k)} : \mathcal{B}^1 \rightarrow [0,\infty]$ be defined by 
\begin{eqnarray}
\mu_{\psi}^{(k)}(A)= (\psi,E^{(k)}(A)\psi).
\nonumber
\end{eqnarray}

\begin{Definition}
Let $\psi_k(t)= e^{-itH^{(k)}}\psi$ and $\psi_k(t,n)= (\delta_n, \psi_k(t))$ 
for $t \in \mathbb{R}$ and $n\in \mathbb{N}$.
Let $a_{\psi}^{(k)}(n,T)$, $\langle |X|^p\rangle_{\psi}^{(k)}(T)$, and $\beta_{\psi}^{(k)}(p)$ be defined by, for $T >0$ and $p>0$,
\begin{eqnarray}
a_{\psi}^{(k)}(n,T)&=&
\cfrac{1}{T}\int_{0}^{\infty} e^{-\frac{t}{T}}|\psi_k(t,n)|^2 dt,
\nonumber
\\
\langle |X|^p\rangle_{\psi}^{(k)}(T)&=&
\sum_{n=1}^{\infty}n^pa_{\psi}^{(k)}(n,T),
\nonumber
\\
\beta_{\psi}^{(k)}(p)&=&
\cfrac{1}{p}\liminf_{T \rightarrow \infty}\cfrac{{\rm log}\langle |X|^p\rangle_{\psi}^{(k)}(T)}{{\rm log}T}.
\nonumber
\end{eqnarray}
\end{Definition}
We call $\beta_{\psi}^{(k)}$ the intermittency function.
The closed subspace $\mathcal{H}_{\psi}$ of $l^2(\mathbb{N})$ is defined by
\begin{eqnarray}
\mathcal{H}_{\psi}=
\overline{
\left\{
p(H^{(k)})\psi \in l^2(\mathbb{N})\mid  \text{$p$ is a polynomial}
\right\}},
\nonumber
\end{eqnarray}
and let $U_{\psi}:\mathcal{H}_{\psi}\rightarrow L^2(\mathbb{R}, d\mu_{\psi}^{(k)})$ 
be defined by
\[
U_{\psi}(p(H^{(k)})\psi)(x)= p(x).
\]

\begin{Lemma}\label{H}
Let $\alpha = {\rm dim}^*(\mu_\psi^{(k)})$ and $\epsilon>0$. 
Then there exists 
$C_1=C_1(\epsilon, \psi)>0$ such that for any $T>0$,
\[
\langle |X|^p\rangle_{\psi}^{(k)}(T) \geq C_1T^{p(\alpha-\epsilon)}.
\]
In particular, for any $p>0$,
\[
{\rm dim}^*(\mu_{\psi}^{(k)}) \leq \beta_{\psi}^{(k)}(p).
\]
\end{Lemma}
\begin{Proof}\rm
We denote 
$\mu_{\psi}^{(k)}$, $a_{\psi}^{(k)}$, $\langle |X|^p \rangle^{(k)}_{\psi}$, and $\psi_{k}(t,n)$ by 
$\mu_{\psi}$, $a_{\psi}$, $\langle |X|^p \rangle_{\psi}$, and $\psi(t,n)$ for simplicity of notation.
Let $\epsilon>0$ and $\gamma:\mathbb{R}\rightarrow \mathbb{R}$ be the local Hausdorff dimension of $\mu_{\psi}$:
\[
\gamma(x)=
\liminf_{\delta \rightarrow 0}
\cfrac{{\rm log}(\mu_{\psi}([x-\delta, x+\delta]))}{{\rm log}\delta}.
\nonumber
\]
By \cite[Chapter 10, Proposition 10.1]{Haus}, we see that $\begin{displaystyle}\mu {\rm \mathchar`-}
\esssup\displaylimits_{x}\gamma(x)={\rm dim}^*\mu_{\psi}=\alpha
\end{displaystyle}$.
Thus there exists $S_{\epsilon}\in \mathcal{B}^1$ such that 
$\mu_\psi(S_{\epsilon})>0$ and 
$\gamma(x)>\alpha-\epsilon$ for $x \in S_{\epsilon}$. 
Let
$\begin{displaystyle}
\gamma_{\delta}(x)=
\inf_{\delta^{\prime}<\delta}\cfrac{{\rm log}\mu_{\psi}([x-\delta^{\prime},x+\delta^{\prime}])}{{\rm log}\delta^{\prime}}
\end{displaystyle}$. 
By Egorov's theorem, there exists 
$S_{\epsilon}^{\prime}\subset S_{\epsilon}$ such that 
$\mu_{\psi}(S_{\epsilon}^{\prime})>0$ and 
$\gamma_{\delta}$ converges uniformly to $\gamma$ on $S_{\epsilon}^{\prime}$. 
Let $\psi^{\prime} = E(S_{\epsilon}^{\prime})\psi$. We see that 
$\|\psi^{\prime}\|^2=\mu_{\psi}(S_{\epsilon}^{\prime})>0$ and 
$\mu_{\psi^{\prime}}$ is uniformly $(\alpha-\epsilon)$-H\"{o}lder continuous.
Let $\chi= \psi-\psi^{\prime}$. Then we see that
\begin{eqnarray}
&&\sum_{n=1}^{N}
a_{\psi}(n,T)
\nonumber\\
&&=
\cfrac{1}{T}
\int_{\mathbb{R}}e^{-\frac{t}{T}}
\sum_{n=1}^{N}|\psi^{\prime}(t,n)+\chi(t,n)|^2dt
\nonumber
\\
&&\leq
\sum_{n=1}^{N}
\cfrac{1}{T}
\int_{\mathbb{R}}e^{-\frac{t}{T}}
|\psi^{\prime}(t,n)|^2dt
+2
\left(
\sum_{n=1}^{N}
\cfrac{1}{T}
\int_{\mathbb{R}}e^{-\frac{t}{T}}
|\psi^{\prime}(t,n)|^2dt
\right)^{\frac{1}{2}}\|\chi\|
+\|\chi\|^2.
\label{eq}
\end{eqnarray}
We assume that $c >0$ and $N\in \mathbb{N}$ satisfy
\begin{eqnarray}
\left(
\sum_{n=1}^{N-1}
\cfrac{1}{T}
\int_{\mathbb{R}}e^{-\frac{t}{T}}
|\psi^{\prime}(t,n)|^2dt
\right)^{\frac{1}{2}}
\leq c\|\psi^\prime\|.
\label{q2e}
\end{eqnarray}
By $(\ref{eq})$ and $(\ref{q2e})$, we see that
\begin{eqnarray}
\sum_{n=1}^{N-1}
a_{\psi}(n,T)\leq \left(
c\|\psi^\prime\|+\|\chi\|
\right)^2.
\nonumber
\end{eqnarray}
Taking $c=-\frac{\|\chi\|}{\|\psi^{\prime}\|}+
\sqrt{\left(\frac{\|\chi\|}{\|\psi^{\prime}\|}\right)^2+\cfrac{1}{2}}$\:,
we have 
\begin{equation}
\sum_{n=1}^{N-1}a_{\psi}(n,T)
%\leq
%\left(c\|\psi^\prime\|+\|\chi\|\right)^2
\leq
 \frac{1}{2}\|\psi^{\prime}\|^2+\|\chi\|^2.
 \label{LK}
\end{equation}
On the other hand, let $C_1^{\prime}=-\frac{\|\chi\|}{\|\psi^{\prime}\|}+
\sqrt{\left(\frac{\|\chi\|}{\|\psi^{\prime}\|}\right)^2+\cfrac{1}{2}}$ and
\begin{equation}
N(T)=
{\rm max}
\left\{
N \in \mathbb{N} \mid
\left(
\sum_{n=1}^{N-1}
\cfrac{1}{T}
\int_{\mathbb{R}}e^{-\frac{t}{T}}
|\psi^{\prime}(t,n)|^2dt
\right)^{\frac{1}{2}}
\leq C_1^{\prime}\|\psi^{\prime}\|
\right\}.
\label{KO8}
\end{equation}
Then $(\ref{q2e})$ holds for $c=C_1^{\prime}$ and $N=N(T)$. 
By $(\ref{LK})$, we have
\begin{equation}
\sum_{n=N(T)}^{\infty}a_{\psi}(n,T)\geq \cfrac{1}{2}\|\psi^{\prime}\|^2.
\label{MNBV}
\end{equation}
%F $(\ref{hint})$,
Note that $\mu_{\psi^\prime}$ is uniformly $(\alpha-\epsilon)$-H\"{o}lder continuous. 
Hence, by Lemma \ref{cor1}, there exists $\widetilde{C}=\widetilde{C}(\alpha-\epsilon, \mu_{\psi^\prime})>0$ such that for any $T>0$,
\begin{equation}
\sum_{n=1}^{N}
\cfrac{1}{T}
\int_{\mathbb{R}}e^{-\frac{t}{T}}
|\psi^{\prime}(t,n)|^2dt
=
\sum_{n=1}^{N}\cfrac{1}{T}\int_{0}^{\infty} e^{-\frac{t}{T}}|\widehat{U_{\psi^\prime}\delta_n \mu_{\psi^\prime}}|^2 dt
\leq
\widetilde{C} NT^{-(\alpha-\epsilon)}.
\label{hint}
\end{equation}
By $(\ref{KO8})$ and $(\ref{hint})$, we have 
\begin{equation}
N(T)\geq \frac{(C_1^{\prime} \|\psi^{\prime}\|)^2}{\widetilde{C}}
T^{\alpha-\epsilon}. 
\label{MNBV1}
\end{equation}
$(\ref{MNBV})$ and $(\ref{MNBV1})$ show that there exists $C_1=C_1(\epsilon, \psi)>0$ such that for any $T>0$,
\begin{eqnarray}
\langle |X|^p\rangle_{\psi}(T)\geq\sum_{n=N(T)}^{\infty}n^pa_{\psi}(n,T)
\geq\cfrac{1}{2}\|\psi^{\prime}\|^2N(T)^p
\geq
C_1T^{p(\alpha-\epsilon)}.
\nonumber
\end{eqnarray}
Moreover, we see that for $\epsilon>0$ and $T^{\prime}>1$,
\begin{eqnarray}
\cfrac{1}{p}
\inf_{T>T^{\prime}}
\cfrac{{\rm log}\langle |X|^p\rangle_{\psi}(T)}{{\rm log}T}\geq
\alpha-\epsilon+
\cfrac{1}{p}
\inf_{T>T^{\prime}}
\cfrac{{\rm log}C_1}{{\rm log}T}
=\alpha-\epsilon.
\nonumber
\end{eqnarray}
This implies our assertion.
\qed
\end{Proof}
%%%%%%%%%%%%%%%%%%%%%%%%%%%%%%%%%%%%%%%%%%%%%%%%%%%%%%%%%%%%%%%%%

\section{Estimates of operator kernel}
In this section we prepare some lemmas to estimate the intermittency function. 
We estimate the operator kernel in Lemma $\ref{6.2}$ and $\ref{6.2.4}$ by using a quadratic form theory and Hellfer-Sj\"{o}strand formula. 
\begin{comment}
Let $\mathcal{H}$ ba a complex Hilbert space.
Let $\mathfrak{t}:\mathcal{H}\times \mathcal{H}\rightarrow \mathbb{C}$ be 
a densely defined, closed, and symmetric form bounded from below, and 
let $T$ be the self-adjoint operator associated with $\mathfrak{t}$.
Suppose that $\mathfrak{s}:\mathcal{H}\times \mathcal{H}\rightarrow \mathbb{C}$ is 
a relatively bounded sesquilinear form with respect to $\mathfrak{t}$ such that 
for any $f \in \mathcal{D}(\mathfrak{t}) \subset \mathcal{D}(\mathfrak{s})$,
\[
|\mathfrak{s}[f]|\leq a\mathfrak{t}[f]+b\|f\|^2, \qquad 0<a<1,\;b\geq0.
\]
Then it is known that $\mathfrak{t}^{\prime}= \mathfrak{s}+\mathfrak{t}$ is sectorial and closed. 
Let $T^{\prime}$ be the m-sectorial operators associated with $\mathfrak{t}^{\prime}$.
If $z\in \rho(T)$ and 
\[
2\|(aT+b)(T-z)^{-1}\|\leq \gamma <1,
\]
then
$z\in \rho (T^{\prime})$ and 
\[
\|(T^{\prime}-z)^{-1}-(T-z)^{-1}\|\leq \cfrac{4\gamma}{(1-\gamma)^2}\|(T-z)^{-1}\|.
\]
\end{comment}

We denote $H^{(k)}$, $a_{k}(n)$, and $d_k(n)$ by 
$H$, $a(n)$, and $d(n)$, respectively for simplicity of notation.
Let $\beta>0$ and let $\mathscr{D}=\{f:\mathbb{N}\rightarrow \mathbb{C}\mid \#{\rm supp}(f)<\infty\}$.
Let $P$, $\Delta$, and $M_{\beta}:l^2(\mathbb{N})\rightarrow l^2(\mathbb{N})$ with its domain $\mathscr{D}$ be defined by
\begin{eqnarray}
P f(n)&=& a(n)f(n+1),
\nonumber
\\
\Delta f&=& (P-I)f,
\nonumber
\\
M_{\beta} f(n)&=& \beta^nf(n),
\nonumber
\end{eqnarray}
and let $T_{\beta}= M_{\beta}^{-1}TM_{\beta}$ for an operator $T:l^2(\mathbb{N})\rightarrow l^2(\mathbb{N})$.

\begin{Lemma}
Let $f\in \mathscr{D}$. It follows that
\begin{enumerate}[$(1)$]
\item
$P^*f(n)=\begin{cases}
0&(n=1),
\label{I}
\\
a(n-1)f(n-1)&(n \geq 2),
\end{cases}
$
\item
\label{II}
$
H^{(k)}f=
\begin{cases}
(\Delta\Delta^*-\delta_1) f&(k=1),\\
\Delta\Delta^* f&(k\geq2),
\end{cases}
$
\item
$\Delta_{\beta}f=(\beta P-I)f$,
\label{III}
\item
$(\Delta^*)_{\beta}f=(\beta^{-1} P^*-I)f$.
\label{IV}
\end{enumerate}
\end{Lemma}
\begin{Proof}\rm
Let $f,g \in \mathscr{D}$. Then we see that 
\begin{eqnarray}
(Pg,f)=\sum_{n=1}^{\infty}a(n)\overline{g(n+1)}f(n)=\sum_{n=2}^{\infty}\overline{g(n)}a(n-1)f(n-1).
\nonumber
\end{eqnarray}
This implies $(\ref{I})$. We see that 
\begin{eqnarray}
\Delta\Delta^* f(n)&=&a(n)\Delta^* f(n+1)-\Delta^* f(n)
\nonumber
\\
&=&
\begin{cases}
a(1)(a(1)f(1)-f(2))+f(1)&(n=1)\\
a(n)(a(n)f(n)-f(n+1))-a(n-1)f(n-1)+f(n)&(n\geq 2)
\end{cases}
\nonumber
\\
&=&
\begin{cases}
-a(1)f(2)+\{a(1)^2+1\}f(1)&(n=1)\\
-a(n)f({n+1})+\{a(n)^2+1\}f(n)-a({n-1})f({n-1})&(n\geq 2).
\end{cases}
\label{098}
\end{eqnarray}
$(\ref{II})$ follows from $(\ref{098})$ and (\ref{4.1}). 
We can prove $(\ref{III})$ and $(\ref{IV})$ straightforwardly.
\qed
\end{Proof}

Let $\beta>0$ and the sesquilinear form $\mathfrak{h}_{\beta}:l^2(\mathbb{N})\times l^2(\mathbb{N})\rightarrow \mathbb{C}$ with its domain $\mathscr{D}$ be defined by
\begin{equation}
\mathfrak{h}_{\beta}(f,g)
=
((\Delta_{\beta})^*f,(\Delta^*)_{\beta}g)
=(f,H_{\beta} \; g).
\nonumber
\end{equation}

\begin{Lemma}
For any $t>0$ and $f\in \mathscr{D}$,
\begin{eqnarray}
|\mathfrak{h}_{\beta}[f]-\mathfrak{h}_1[f]|
\leq C(\beta)\;\cfrac{t}{2}\;
\mathfrak{h}_1[f]
+C(\beta)(1+\cfrac{1}{2t}\:)\|f\|^2,
\label{5.2.2}
\end{eqnarray}
where $C(\beta)=|\beta-1|+|\beta^{-1}-1|=|\beta-\beta^{-1}|$.
\end{Lemma}

\begin{Proof}\rm
Let $t>0$ and $f \in \mathscr{D}$. Then we see that
\begin{eqnarray}
|\mathfrak{h}_{\beta}[f]-\mathfrak{h}_1[f]|&=&
|((\beta P^*-I)f,(\beta^{-1}P^*-I)f)-((P-I)f,(P^*-I)f)|
\nonumber
\\
&\leq&
|\beta-1||(f,Pf)|
+|\beta^{-1}-1||(f,P^*f)|
\nonumber
\\
&\leq&
|\beta-1|\{
(\Delta^*f,f)+(f,f)
\}
+|\beta^{-1}-1|
\{
(f,\Delta^*f)+(f,f)
\}
\nonumber
\\
&\leq&
C(\beta)\|f\|\|\Delta^*f\|+C(\beta)\|f\|^2
\nonumber
\\
&\leq&
C(\beta)
\left(
\cfrac{t}{2}\|\Delta^*f\|^2+\cfrac{1}{2t}\|f\|^2+\|f\|^2
\right).
\nonumber
\end{eqnarray}
\qed
\end{Proof}

\begin{Lemma}\label{6.2}
Let $z\in \mathbb{C}^+\coloneqq \{z\in \mathbb{C}\mid \text{\rm Im}z>0\}$. Let $0<\gamma<1$, and let $\eta_z$,  $m_z$, and $\alpha_z$ be 
\begin{eqnarray}
\eta_z&=& {\rm dist}(z, \sigma(H)),
\nonumber
\\
m_z&=& \cfrac{\eta_z}{\sqrt{\eta_z+|z|}+1},
\nonumber
\\
\alpha_z(\gamma)&=& \cfrac{1}{4}(\gamma m_z+\sqrt{(\gamma m_z)^2+16}).
\nonumber
\end{eqnarray}
Then for any $i,j \in \mathbb{N}$,
\begin{equation}
%\label{main ker}
|(\delta_i, (H-z)^{-1}\delta_j)|\leq
{\alpha_z(\gamma)}^{-|i-j|}
\cfrac{1}{\eta_z}
\left(
\cfrac{1+\gamma}{1-\gamma}
\right)^2.
\nonumber
\end{equation}
\end{Lemma}
\begin{Proof}\rm
It follows from $\|(H-z)^{-1}\|=\eta_z^{-1}$ that for any $t>0$
\begin{eqnarray}
2
\left\|
C(\beta)
\left(
\cfrac{t}{2}H+1+\frac{1}{2t}
\right)
(H-z)^{-1}
\right\|
\leq
C(\beta)
\left\{
\left(
1+\cfrac{|z|}{\eta_z}
\right)t+
\left(
2+\cfrac{1}{t}
\right)
\cfrac{1}{\eta_z}
\right\}
.
\label{Y}
\end{eqnarray}
Let $\gamma \in (0,1)$ and 
\begin{eqnarray}
t_z
%=\sqrt{\left(1+\frac{|z|}{\eta_z}\right)\frac{1}{\eta_z}}
&=&\cfrac{1}{\sqrt{\eta_z+|z|}},
\nonumber
\\
%\begin{displaystyle}
\beta_z&=&
\frac{1}{4}(\gamma m_z+\sqrt{(\gamma m_z)^2+16})>1.
\nonumber
%\end{displaystyle}
\end{eqnarray}
By the inequality of arithmetic and geometric means, we see that for any $z\in \mathbb{C^+},$
\begin{eqnarray}%\label{ker1}
C(\beta_z)
\left\{
\left(
1+\cfrac{|z|}{\eta_z}
\right)t_z+
\left(
2+\cfrac{1}{t_z}
\right)
\cfrac{1}{\eta_z}
\right\}
&=&
C(\beta_z)
\left\{
\left(
1+\cfrac{|z|}{\eta_z}
\right)t_z+
\cfrac{1}{\eta_z}\cfrac{1}{t_z}
+
\cfrac{2}{\eta_z}
\right\}
\nonumber
\\
&=&
2\cfrac{\sqrt{\eta_z+|z|}+1}{\eta_z}\:C(\beta_z)
\nonumber
\\
&=&
\cfrac{2}{m_z}
\left(
\beta_z-\frac{1}{\beta_z}
\right)
\nonumber
\\
&=& \gamma.
\label{QAZ}
\end{eqnarray}
$(\ref{Y})$ and $(\ref{QAZ})$ imply that for any $z\in \mathbb{C}^+$,
\begin{eqnarray}
2
\left\|
C(\beta_z)
\left(
\cfrac{t_z}{2}H+1+\frac{1}{2t_z}
\right)
(H-z)^{-1}
\right\|
\leq \gamma
\label{ghj}
\end{eqnarray} 
By $(\ref{5.2.2})$, $(\ref{ghj})$ and Lemma \ref{quadratic thm},  
there exists the m-sectoral operator $H_{\beta_z}$ associated with $\mathfrak{h}_{\beta_z}$ and for any $z \in \mathbb{C}^{+}$, 
\begin{eqnarray}%\label{ker3}
\|(H_{\beta_z}-z)^{-1}-(H-z)^{-1}\|\leq 
\cfrac{4\gamma}{(1-\gamma)^2}\|(H-z)^{-1}\|.
\nonumber
\end{eqnarray}
Therefore we see that
\begin{equation}%\label{ker4}
\|(H_{\beta_z}-z)^{-1}\|\leq 
\cfrac{1}{\eta_z}
\left(
\cfrac{1+\gamma}{1-\gamma}
\right)^2.
\nonumber
\end{equation}
Let $i,j \in \mathbb{N}$ with $i<j$. Then we see that
\begin{eqnarray}
|(\delta_i,(H-z)^{-1}\delta_j)|
&=&
|(M_{\beta_z}\delta_i,(H_{\beta_z}-z)^{-1}M_{\beta_z}^{-1}\delta_j)|
\nonumber
\\
&\leq&
{\beta_z}^{i-j}\|(H_{\beta}-z)^{-1}\|
\nonumber
\\
&\leq&
{\beta_z}^{-|i-j|}
\cfrac{1}{\eta_z}
\left(
\cfrac{1+\gamma}{1-\gamma}
\right)^2.
\label{kjh}
\end{eqnarray}
This implies our assertion in the case of $i<j$. In the case of $i \geq j$, let
\[
\beta_z=
\frac{1}{4}(-\gamma m_z+\sqrt{(\gamma m_z)^2+16}).
\] 
Then we can prove $(\ref{kjh})$ similarly.
\qed
\end{Proof}
Let $f\in C^{n}(\mathbb{R})$, and the norm $\opnorm{\cdot}_n$ on $C^{n}(\mathbb{R})$ be defined by
\[
\opnorm{f}_n=
\sum_{r=0}^{n}
\int_{\mathbb{R}}|f^{(r)}(x)|\langle x \rangle^{r-1}dx.
\]
The next lemma is used to estimate the intermittency function.
\begin{Lemma}\label{6.2.4}
Suppose that $f \in C^{2k+3}(\mathbb{R})$ and $\opnorm{f}_{2k+3}<\infty$.
Then there exists $C_2=C_2(k)>0$ such that for any $i,j\in \mathbb{N}$,
\[
|(\delta_i, f(H)\delta_j)|\leq C_2\opnorm{f}_{2k+3}\langle i-j \rangle^{-k},
\]
where $\langle x \rangle = (1+|x|^2)^{\frac{1}{2}}$.
\end{Lemma}
\begin{Proof}\rm
Let $n \geq 0$ and 
$\tau \in C^{\infty}_0(\mathbb{R})$ such that 
$\begin{displaystyle}
\tau(x)=
\begin{cases}
1&(|x|\leq 1)\\
0&(|x|\geq 2)
\end{cases}
\end{displaystyle}$.
By Helffer-Sj\"{o}strand formula \cite[2.2 The Helffer-Sj\"{o}strand formula]{Hollfer}, we see that 
\begin{eqnarray}
f(H)=\cfrac{1}{\pi}
\int_{\mathbb{C}}
\cfrac{\partial \tilde{f}}{\partial \overline{z}}(z)(H-z)^{-1}dxdy,
\nonumber
\\
\tilde{f}(z)=
\left\{
\sum_{r=0}^{n}
f^{(r)}(x)\cfrac{(iy)^r}{r!}
\right\}
\tau
\left(\frac{y}{\langle x \rangle}
\right).
\nonumber
\end{eqnarray}
We see that
\begin{eqnarray}
\left|
\cfrac{\partial \tilde{f}}{\partial \overline{z}}(z)
\right|
%&=&
%\left|
%\cfrac{1}{2}
%\left(
%\cfrac{\partial}{\partial x}+i\cfrac{\partial}{\partial y}
%\right)
%\tilde{f}(z)
%\right|
%\nonumber
%\\
&\leq&\;
\cfrac{1}{2}
\left|
f^{(n+1)}(x)
\cfrac{(iy)^n}{n!}
\tau
\left(
\cfrac{y}{\langle x \rangle}
\right)
\right|
\nonumber
\\
&
+&
\cfrac{1}{2}
\left|
%\left(
\sum_{r=0}^{n}
f^{(r)}(x)\cfrac{(iy)^r}{r! \langle x \rangle}
%\right)
\right|
\left|
\left(1+xy\langle x \rangle^{-1}
\right)
\tau^{\prime}
\left(
\cfrac{y}{\langle x \rangle}
\right)
\right|.
%\label{ker5}
\nonumber
\end{eqnarray}
Let $A$, $B \subset \mathbb{R}^2$ be
\begin{eqnarray}
A&=&
\left\{
(x,y) \in \mathbb{R}^2
\mid
\left|
\frac{y}{\langle x \rangle}
\right|
\leq 2
\right\}
%\nonumber
, 
B=
\left\{
(x,y) \in \mathbb{R}^2
\mid
1 \leq
\left|
\frac{y}{\langle x \rangle}
\right|
\leq 2
\right\}.
\nonumber
\end{eqnarray}
%$(\ref{ker5})$, 
Then we see that there exists $C_2^{\prime}=C_2^{\prime}(\tau)>0$ such that 
\begin{eqnarray}
\left|
\cfrac{\partial \tilde{f}}{\partial \overline{z}}(z)
\right|
\leq
\cfrac{1}{2}
\left|
f^{(n+1)}(x)
\cfrac{(iy)^n}{n!}
\right|
\mathbbm{1}_{A}(x,y)
+
C_2^{\prime}
\left|
%\left(
\sum_{r=0}^{n}
f^{(r)}(x)\cfrac{(iy)^r}{r! \langle x \rangle}
%\right)
\right|
\mathbbm{1}_{B}(x,y).
\nonumber
\end{eqnarray}
Let $i, j \in \mathbb{N}$. Then we have 
\begin{eqnarray}
|(\delta_i,f(H)\delta_j)|
&
\leq
&
\cfrac{1}{2\pi}
\int_{\mathbb{C}}
\left|
f^{(n+1)}(x)
\cfrac{(iy)^n}{n!}
\right|
|
(\delta_i, (H-z)^{-1} \delta_j)
|
\mathbbm{1}_{A}(x,y)
dxdy
\nonumber
\\
&+&
\cfrac{C_2^{\prime}}{\pi}
\int_{\mathbb{C}}
\left|
%\left(
\sum_{r=0}^{n}
f^{(r)}(x)\cfrac{(iy)^r}{r! \langle x \rangle}
%\right)
\right|
|
(\delta_i, (H-z)^{-1} \delta_j)
|
\mathbbm{1}_{B}(x,y)
dxdy.
\label{ker17}
\end{eqnarray}
Let $\gamma_z= \cfrac{1}{\sqrt{\eta_z+|z|}+1}<1$ and 
$\alpha_z=\alpha_z(\gamma_z)$.
By Lemma $\ref{6.2}$, we see that
\begin{eqnarray}
&&
\int_{\mathbb{C}}
\left|
%\left(
\sum_{r=0}^{n}
f^{(r)}(x)\cfrac{(iy)^r}{r! \langle x \rangle}
%\right)
\right|
|
(\delta_i, (H-z)^{-1} \delta_j)
|
\mathbbm{1}_{B}(x,y)
dxdy
\nonumber
\\
&&\leq
\sum_{r=0}^{n}
\int_{\mathbb{C}}
\left|
%\left(
f^{(r)}(x)\cfrac{(iy)^r}{r! \langle x \rangle}
%\right)
\right|
\alpha_z^{-|i-j|}
\cfrac{1}{\eta_z}
\left(
\cfrac{1+\gamma_z}{1-\gamma_z}
\right)^2
\mathbbm{1}_{B}(x,y)
dxdy.
\label{ker6}
\end{eqnarray}
We estimate the lower bound of $\alpha_z$.
Suppose that $(x,y)\in B$, then $1\leq |y|\leq \eta_z$ and $|z|\leq \sqrt{2}|y|$. 
Therefore
\begin{eqnarray}
\gamma_z m_z
=\cfrac{\eta_z}{
\eta_z+|z|+1
+2\sqrt{\eta_z+|z|}
}
\geq
\cfrac{1}{2+\sqrt{2}+2\sqrt{1+\sqrt{2}}}
\label{lkj}
\end{eqnarray}
Let 
$b=\cfrac{1}{2+\sqrt{2}+2\sqrt{1+\sqrt{2}}}$. 
By the definition of $\alpha_z$, 
$(\ref{lkj})$ implies that 
\begin{eqnarray}
\alpha_z \geq 
\cfrac{1}{4}
\left(
b+\sqrt{b^2+16}
\right)
>1.
\label{ker8}
\end{eqnarray}
Let $B=\cfrac{1}{4}
\left(
b+\sqrt{b^2+16}
\right).$
%Let $(x,y)\in B$, then 
We see that
\begin{eqnarray}
\cfrac{1+\gamma_z}{1-\gamma_z} 
=1+
\cfrac{2}{\sqrt{\eta_z+|z|}}
\leq1+\sqrt{2}.
\label{ker9}
\end{eqnarray}
%$(\ref{ker6})$, $(\ref{ker8})$, and $(\ref{ker9})$ mean that
By $(\ref{ker6})$, $(\ref{ker8})$, and $(\ref{ker9})$, we see that
\begin{eqnarray}
&&
\int_{\mathbb{C}}
\left|
%\left(
\sum_{r=0}^{n}
f^{(r)}(x)\cfrac{(iy)^r}{r! \langle x \rangle}
%\right)
\right|
|
(\delta_i, (H-z)^{-1} \delta_j)
|
\mathbbm{1}_{B}(x,y)
dxdy
\nonumber
\\
&&
\leq
\left(
1+\sqrt{2}
\right)^2
B^{-|i-j|}
\sum_{r=0}^{n}
\cfrac{1}{r!}
\int_{\mathbb{C}}
|f^{(r)}(x)|
\left|
\cfrac{y^{r-1}}{\langle x \rangle}
\right|
\mathbbm{1}_{B}(x,y)
dxdy
\nonumber
\\
&&
\leq
\left(
1+\sqrt{2}
\right)^2
B^{-|i-j|}
\sum_{r=0}^{n}
\cfrac{2^{r-1}}{r!}
\int_{\mathbb{C}}
|f^{(r)}(x)|
\langle x \rangle^{r-2}
\mathbbm{1}_{B}(x,y)
dxdy
\nonumber
\\
&&
\leq
\left(
1+\sqrt{2}
\right)^2
B^{-|i-j|}
\sum_{r=0}^{n}
\int_{\mathbb{R}}
|f^{(r)}(x)|
\langle x \rangle^{r-1}
dx.
\label{ker17.1}
\end{eqnarray}
By Lemma $\ref{6.2}$, we see that
\begin{eqnarray}
&&
\int_{\mathbb{C}}
\left|
f^{(n+1)}(x)
\cfrac{(iy)^n}{n!}
\right|
|
(\delta_i, (H-z)^{-1} \delta_j)
|
\mathbbm{1}_{A}(x,y)
dxdy
\nonumber
\\
&&
\leq
\int_{\mathbb{C}}
\left|
f^{(n+1)}(x)
\cfrac{(iy)^n}{n!}
\right|
\alpha_z^{-|i-j|}
\cfrac{1}{\eta_z}
\left(
\cfrac{1+\gamma_z}{1-\gamma_z}
\right)^2
\mathbbm{1}_{A}(x,y)
dxdy.
\label{kerq}
\end{eqnarray}
Note that for any $k \in \mathbb{Z}_{\geq 0}$ and $t>0$,% then it follows that 
\begin{equation}
e^{-t}\leq  
\cfrac{e^{-k}k^k}{t^k}.
\nonumber
\end{equation}
This implies that, for $i,j \in \mathbb{N}$ with $i\neq j$,
\begin{eqnarray}
\alpha_z^{-|i-j|}
\leq
 \left(
1+
\left(
\frac{\gamma_z m_z}{4}
\right)^2
\right)^{-\frac{|i-j|}{2}}
\leq
\cfrac{e^{-k}(2k)^k}{
|i-j|^k
\left(
{\rm log}
\left(
1+
\left(
\frac{\gamma_z m_z}{4}
\right)^2
\right)
\right)^k
}.
\label{ker11}
\end{eqnarray}
Suppose that $(x,y) \in A$, then 
$|y|\leq 2\langle x\rangle$ and $|z|\leq \sqrt{5}\langle x \rangle$. We see that
\begin{eqnarray}
\gamma_z m_z
\geq
\cfrac{1}{2}\:
\cfrac{\eta_z}{\eta_z+|z|+1}
\geq
\cfrac{1}{2}\:
\cfrac{|y|}{|y|+|z|+1}
\geq 
\cfrac{3-\sqrt{5}}{8}\:
\cfrac{|y|}{\langle x \rangle}.
\label{ker12}
\end{eqnarray}
%$(\ref{ker11})$ and $(\ref{ker12})$ show that
$(\ref{ker11})$ and $(\ref{ker12})$ imply
\begin{equation}
\alpha_z^{-|i-j|}
\leq
\cfrac{e^{-k}(2k)^k}{
|i-j|^k
\left(
{\rm log}
\left(
1+
\left(
\frac{3-\sqrt{5}}{32}
\;
\frac{|y|}{\langle x \rangle}
\right)^2
\right)
\right)^k
}.
\label{ker13}
\end{equation}
We see that
\begin{equation}
\cfrac{1+\gamma_z}{1-\gamma_z}\leq
1+\sqrt{\frac{2}{|y|}}.
\label{ker14}
\end{equation}
%$(\ref{kerq})$, $(\ref{ker13})$, and $(\ref{ker14})$ mean that
By $(\ref{kerq})$, $(\ref{ker13})$, and $(\ref{ker14})$, we see that
\begin{eqnarray}
&&
\int_{\mathbb{C}}
\left|
f^{(n+1)}(x)
\cfrac{(iy)^n}{n!}
\right|
|
(\delta_i, (H-z)^{-1} \delta_j)
|
\mathbbm{1}_{A}(x,y)
dxdy
\nonumber
\\
&&\leq
\cfrac{2e^{-k}(2k)^k}{n! |i-j|^k}
\int_{\mathbb{C}}
\cfrac{|f^{(n+1)}(x)||y|^{n-1}}{
\left(
{\rm log}
\left(
1+
\left(
\frac{3-\sqrt{5}}{32}
\;
\frac{|y|}{\langle x \rangle}
\right)^2
\right)
\right)^k
}
\left(
1+\frac{2}{|y|}
\right)
\mathbbm{1}_{A}(x,y)
dxdy
\nonumber
\\
&&\leq
\cfrac{8e^{-k}(2k)^k}{n! |i-j|^k}
\int_{\mathbb{R}}
dx|f^{(n+1)}(x)|
\int_{0}^{2\langle x \rangle}dy
\cfrac{|y|^{n-1}+|y|^{n-2}}{
\left(
{\rm log}
\left(
1+
\left(
\frac{3-\sqrt{5}}{32}
\;
\frac{|y|}{\langle x \rangle}
\right)^2
\right)
\right)^k
}
\nonumber
\\
&&\leq
\cfrac{8e^{-k}(2k)^k}{n! |i-j|^k}
\int_{\mathbb{R}}
dx|f^{(n+1)}(x)|
\langle x \rangle^{n}
\int_{0}^{2}dt
\cfrac{t^{n-1}+t^{n-2}}{
\left(
{\rm log}
\left(
1+
\left(
\frac{3-\sqrt{5}}{32}
\;
t
\right)^2
\right)
\right)^k
}.
\label{ker15}
\end{eqnarray}
If $n>2k+1$, 
\begin{eqnarray}
C_2^{\prime \prime}(n)
\coloneqq
\int_{0}^{2}dt
\cfrac{t^{n-1}+t^{n-2}}
{
\left(
{\rm log}
\left(
1+
\left(
\frac{3-\sqrt{5}}{32}
\;
t
\right)^2
\right)
\right)^k
}<\infty.
\label{ker16}
\end{eqnarray}
Let $n=2k+2$. Then $(\ref{ker17})$, $(\ref{ker17.1})$, $(\ref{ker15})$, and $(\ref{ker16})$ imply that there exists $C_2=C_2(k)>0$ such that for any $i,j \in \mathbb{N}$,
\begin{eqnarray}
&&|(\delta_i,f(H)\delta_j)|
\nonumber
\\
&&\leq
\cfrac{C_2^{\prime}}{\pi}
\left(
1+\sqrt{2}
\right)^2
B^{-|i-j|}
\opnorm{f}_{2k+2}
+
\cfrac{1}{2\pi}
C_2^{\prime\prime}(2k+2)
\cfrac{8e^{-k}(2k)^k}{(2k+2)! |i-j|^k}
\opnorm{f}_{2k+3}
\nonumber
\\
&&\leq
C_2\opnorm{f}_{2k+3}\langle i-j \rangle^{-k}.
\nonumber
\end{eqnarray}
This implies our assertion.
\qed
\end{Proof}

\section{Intermittency function and proof of the main result}
In this section, we mainly consider the distribution of $a_{\psi}^{(k)}(n,T)$ and estimate the lower and upper bounds the momentum $\langle |X|^p\rangle_{\psi}^{(k)}(T)$. 
From this, we calculate the intermittency function exactly. 
Finally, we prove Theorem $\ref{main}$ by using the intermittency function.

\subsection{Lower bound of intermittency function}
Let $k\in \mathbb{N}$, $\psi \in l^2(\mathbb{N})$, and $T > 0$. We define for $S\in2^{\mathbb{N}}$,
\begin{eqnarray}
P_{\psi}^{(k)}(S, T)= \sum_{n \in S}^{}a_{\psi}^{(k)}(n,T).
\nonumber
\end{eqnarray}
For $M \geq N \geq1$, let subsets $\{N \sim M \}$ and $\{M \sim\infty \}$ of $\mathbb{N}$ be 
\begin{eqnarray}
\{N \sim M \}&=&\{n\in \mathbb{N} \mid N \leq  n\leq M\},
\nonumber
\\ 
\{M \sim\infty \}&=&\{n\in \mathbb{N} \mid n\geq M\}.
\nonumber
\end{eqnarray}

\begin{Lemma}\label{6.1.1}
Let $T>0$ and $\epsilon>0$. 
Suppose that $B\in \mathcal{B}^1$ and $A \coloneqq \mu_{\psi}^{(k)}(B)>0$. 
Let 
\begin{eqnarray}
M_T&=&
\cfrac{A^2}{16J^{(k)}_{\psi}(T^{-1},B)},
\nonumber
\\
J_{\psi}^{(k)}(\epsilon, B)&=& \int_{B}\mu_{\psi}^{(k)}(dx)
\int_{\mathbb{R}}\mu_{\psi}^{(k)}(dy)
\cfrac{\epsilon^2}{(x-y)^2+\epsilon^2}.
\nonumber
\end{eqnarray}
Then for any $T>0$
\[
P_{\psi}^{(k)}(\:\{M_T \sim \infty\}, T)\geq \cfrac{A}{2}>0.
\]
\end{Lemma}
\begin{Proof}\rm
We denote $P_{\psi}^{(k)}$, $\mu_{\psi}^{(k)}$, and $J_{\psi}^{(k)}$ by $P_{\psi}$, $ \mu_{\psi}$, and $J_{\psi}$, respectively for simplicity of notation.
%Abbreviate $P_{\psi}^{(k)}$, $\mu_{\psi}^{(k)}$, and $J_{\psi}^{(k)}$ to $P_{\psi}$, $ \mu_{\psi}$, and $J_{\psi}$.
Let $\rho = E^{(k)}(B)\psi$ and $\chi = \psi - \rho$. 
Note that $\rho \neq 0$. We see that
\begin{eqnarray}
P_{\psi}(\{1 \sim M\}, T)&=&
\sum_{n=1}^{M}\cfrac{1}{T}\int_{\mathbb{R}}dt\:e^{-\frac{t}{T}}|\chi(n,t)+\rho(n,t)|^2
\nonumber
\\
&=&
P_{\chi}(\{1 \sim M\},T)+
P_{\rho}(\{1 \sim M\},T)
+2\sum_{n=1}^{M}\cfrac{1}{T}\int_{\mathbb{R}}dt\:e^{-\frac{t}{T}}
{\rm Re}(\chi(n,t) \overline{\rho(n,t)})
\nonumber
\\
&=&
P_{\chi}(\{1 \sim M\},T)-
P_{\rho}(\{1 \sim M\},T)
+2\sum_{n=1}^{M}\cfrac{1}{T}\int_{\mathbb{R}}dt\:e^{-\frac{t}{T}}
{\rm Re}(\psi(n,t)\overline{\rho(n,t)})
\nonumber
\\
&\leq&
P_{\chi}(\{1 \sim M\},T)
+2\sum_{n=1}^{M}\cfrac{1}{T}\int_{\mathbb{R}}dt\:e^{-\frac{t}{T}}
{\rm Re}
(\psi(n,t)\overline{\rho(n,t)}).
\label{OIU}
\end{eqnarray}
Since $P_{\psi}(\{M \sim \infty \},T)=\|\psi\|^2-P_{\psi}(\{1,M-1\},T)$ and $\|\psi\|^2=\|\rho\|^2+\|\chi\|^2$,
$( \ref{OIU})$ implies that
\begin{equation}
P_{\psi}(\{M \sim \infty \},T)\geq
\|\rho\|^2-2|D(M-1,T)|,
\label{7.2}
\end{equation}
where 
\begin{equation}
D(M,T)=
\sum_{n=1}^{M}\cfrac{1}{T}\int_{\mathbb{R}}dt\:e^{-\frac{t}{T}}
\psi(n,t)\overline{\rho(n,t)}
=
\sum_{n=1}^{M}\cfrac{1}{T}\int_{\mathbb{R}}dt\:e^{-\frac{t}{T}}
(\delta_n,\psi(t))(\rho(t),\delta_n).
\nonumber
\end{equation}
Since $U_{\psi}:\mathcal{H}_{\psi}\rightarrow L^2(\mathbb{R}, d\mu_{\psi}^{(k)})$ is unitary, by Schwarz inequality we see that
\begin{eqnarray}
&&|D(M,T)|
\nonumber
\\
&&=
\left|
\sum_{n=1}^{M}
\cfrac{1}{T}\int_{\mathbb{R}}dt\; e^{-\frac{t}{T}}
\int_{\mathbb{R}}\mu_{\psi}(dx)
\int_{B}\mu_{\psi}(dy)
e^{-it(x-y)}
\overline{U_{\psi}\delta_n(x)}U_{\psi}\delta_n(y)
\right|
\nonumber
\\
&&=
\left|
\int_{\mathbb{R}}\mu_{\psi}(dx)
\int_{B}\mu_{\psi}(dy)\:
\cfrac{1-iT(x-y)}{1+T^2(x-y)^2}
\:
\sum_{n=1}^{M}
\overline{U_{\psi}\delta_n(x)}U_{\psi}\delta_n(y)
\right|
\nonumber
\\
&&\leq
\int_{\mathbb{R}}\mu_{\psi}(dx)
\int_{B}\mu_{\psi}(dy)\:
\cfrac{1}{\sqrt{1+T^2(x-y)^2}}
\left|
\sum_{n=1}^{M}
\overline{U_{\psi}\delta_n(x)}U_{\psi}\delta_n(y)
\right|
\nonumber
\\
&&\leq
%$\begin{displaystyle}
\int_{B}\mu_{\psi}(dy)
\left(
\int_{\mathbb{R}}
\cfrac{\mu_{\psi}(dx)}{1+T^2(x-y)^2}
\right)^{\frac{1}{2}}
\left(
\int_{\mathbb{R}}
\mu_{\psi}(dx)
\left|
\sum_{n=1}^{M}
\overline{U_{\psi}\delta_n(x)}U_{\psi}\delta_n(y)
\right|^2
\right)^{\frac{1}{2}}.
%\end{displaystyle}$
\nonumber
\\
&&\leq
\left(
\int_{B}\mu_{\psi}(dy)
\int_{\mathbb{R}}
\cfrac{\mu_{\psi}(dx)}{1+T^2(x-y)^2}
\right)^{\frac{1}{2}}
\left(
\int_{B}\mu_{\psi}(dy)
\int_{\mathbb{R}}
\mu_{\psi}(dx)
\left|
\sum_{n=1}^{M}
\overline{U_{\psi}\delta_n(x)}U_{\psi}\delta_n(y)
\right|^2
\right)^{\frac{1}{2}}.
\label{7.3}
\end{eqnarray}
Since $U_{\psi}:\mathcal{H}_{\psi}\rightarrow L^2(\mathbb{R}, d\mu_{\psi}^{(k)})$ is unitary, we have
\begin{eqnarray}
\int_{\mathbb{R}}
\mu_{\psi}(dx)
\left|
\sum_{n=1}^{M}
U_{\psi}\delta_n(x)
\overline{
U_{\psi}\delta_n(y)}
\right|^2
&=&
\int_{\mathbb{R}}
\mu_{\psi}(dx)
\left|
U_{\psi}\left(
\sum_{n=1}^{M}
\overline{
U_{\psi}\delta_n(y)}
\delta_n
\right)
(x)
\right|^2
\nonumber
\\
&=&
\left\|
\sum_{n=1}^{M}
\overline{
U_{\psi}\delta_n(y)}
\delta_n
\right\|^2_{l^2(\mathbb{N})}
\nonumber
\\
&=&
\sum_{n=1}^{M}
\left|
U_{\psi}\delta_n(y)
\right|^2.
\label{7.4}
\end{eqnarray}
By $(\ref{7.3})$ and $(\ref{7.4})$, we see that
\begin{eqnarray}
|D(M-1,T)|^2&\leq&
\left(
\int_{B}\mu_{\psi}(dy)
\int_{\mathbb{R}}
\cfrac{\mu_{\psi}(dx)}{1+T^2(x-y)^2}
\right)
\left(
\sum_{n=1}^{M}
\int_{B}\mu_{\psi}(dy)
\left|
U_{\psi}\delta_n(y)
\right|^2
\right)
\nonumber
\\
&\leq&
J_{\psi}(T^{-1},B)\sum_{n=1}^{M}
\|U_{\psi}\delta_n\|_{l^2}^2
\nonumber
\\
&\leq&
MJ_{\psi}(T^{-1},B)
\nonumber
\end{eqnarray}
Let $M=M_T$. Then 
\begin{equation}
|D(M_T-1,T)|\leq
\sqrt{ M_TJ_{\psi}(T^{-1},B)}
=
\cfrac{\|\rho\|^2}{4}.
\label{7.6}
\end{equation}
By $(\ref{7.2})$ and $(\ref{7.6})$, we obtain that
\begin{equation}
P_{\psi}(\{ M_T\sim \infty\}, T) \geq \frac{\|\rho \|^2}{2}=\frac{A}{2}.
\nonumber
\end{equation}
This implies our assertion.
\qed
\end{Proof}
For $\psi \in l^2(\mathbb{N})$,
let an analytic function $m_{\psi}^{(k)}: \mathbb{C}^+ \rightarrow \mathbb{C}^+$ be defined by
\begin{eqnarray}
m_{\psi}^{(k)}(z) = \int_{\mathbb{R}}\cfrac{\mu_{\psi}^{(k)}(d\lambda)}{\lambda-z}
=(\psi, (H^{(k)}-z)^{-1}\psi).
\nonumber
\end{eqnarray}
Let $\epsilon>0$ and $B\in \mathcal{B}^1$. 
We define
\[
I_{\psi}^{(k)}(\epsilon,B)=\epsilon
\int_{B}dE |{\rm Im}\: m_{\psi}^{(k)}(E+i\epsilon)|^2.
\]
\begin{Lemma}\label{pro612}
Let $B=[a,b]\subset \mathbb{R}$. Then there exists $C_3=C_3(a,b)>0$ such that 
for any $\epsilon \in(0,1)$
\begin{equation}
J_{\psi}^{(k)}(\epsilon, B)\leq C_3 I_{\psi}^{(k)}(\epsilon,B).
\label{RDX}
\end{equation}
\end{Lemma}
\begin{Proof}\rm
We denote $J_{\psi}^{(k)}$, $I_{\psi}^{(k)}$, and $\mu_{\psi}^{(k)}$ by 
$J_{\psi}$, $I_{\psi}$, and $\mu_{\psi}$, respectively for simplicity of notation. We see that
\begin{eqnarray}
I_{\psi}(\epsilon,B)
&=&
\epsilon^3\int_{B}dE
\left(
\int_{\mathbb{R}}
\cfrac{\mu_{\psi}(dx)}{\epsilon^2+(E-x)^2}
\right)^2
\nonumber
\\
&=&
\epsilon^3\int_{B}dE
\int_{\mathbb{R}}
\cfrac{\mu_{\psi}(dx)}{\epsilon^2+(E-x)^2}
\int_{\mathbb{R}}
\cfrac{\mu_{\psi}(dy)}{\epsilon^2+(E-y)^2}
\nonumber
\\
&\geq&
%$\begin{displaystyle}
\int_{B}\mu_{\psi}(dx)
\int_{\mathbb{R}}\mu_{\psi}(dy)\:
\epsilon^3
\int_B
\cfrac{dE}{(\epsilon^2+(E-x)^2)(\epsilon^2+(E-y)^2)}.
%\end{displaystyle}$
\label{7.7}
\end{eqnarray}
Let $s= \frac{x-y}{\epsilon}$. Since 
$x \in B=[a,b]$ and $0<\epsilon<1$, we have
\begin{eqnarray}
\epsilon^3
\int_B
\cfrac{dE}{(\epsilon^2+(E-x)^2)(\epsilon^2+(E-y)^2)}
&=&
\int_{\frac{a-x}{\epsilon}}^{\frac{b-x}{\epsilon}}\cfrac{dt}{(1+t^2)(1+(t+s)^2)}
\nonumber
\\
&\geq&
\int_{a-x}^{b-x}\cfrac{dt}{(1+t^2)(1+(|t|+|s|)^2)}.
\nonumber
%\label{7.8}
\end{eqnarray}
If $|s|\leq1$, then
\begin{eqnarray}
\int_{a-x}^{b-x}\cfrac{dt}{(1+t^2)(1+(|t|+|s|)^2)}
&\geq&
\int_{a-x}^{b-x}\cfrac{dt}{(1+t^2)(1+(|t|+1)^2)}.
%\nonumber
\label{TY}
%\\
%&\geq&
%\inf_{x \in B}\int_{a-x}^{b-x}\cfrac{dt}{(1+t^2)(1+(|t|+1)^2)}.
%\nonumber
%\\
%&=&C^{\prime}(a,b)>0.
%\nonumber
%\label{7.9}
\end{eqnarray}
If $|s|\geq 1$, then there exists $C_3^{\prime}=C_3^{\prime}(a,b)>0$ such that for any $x\in B$,
\begin{eqnarray}
\int_{a-x}^{b-x}\cfrac{dt}{(1+t^2)(1+(|t|+|s|)^2)}
&=&
\int_{a-x}^{b-x}
dt\:
\left(
\cfrac{1}{1+t^2}-
\cfrac{1}{1+(|t|+|s|)^2}
\right)
((|t|+|s|)^2-t^2)^{-1}
\nonumber
\\
&\geq& 
\int_{a-x}^{b-x}
dt\:
\left(
\cfrac{1}{1+t^2}-
\cfrac{1}{1+(|t|+1)^2}
\right)
\cfrac{C_3^{\prime}}{1+s^2}
\label{7.10}
\end{eqnarray}
By $(\ref{TY})$ and $(\ref{7.10})$, there exists $C_3=C_3(a,b)>0$ such that for any $x\in B$ and any $y\in \mathbb{R}$, 
\begin{equation}
\epsilon^3
\int_B
\cfrac{dE}{(\epsilon^2+(E-x)^2)(\epsilon^2+(E-y)^2)}
\geq
\cfrac{C_3}{1+s^2},
\qquad s=\frac{x-y}{\epsilon} .
\label{7.11}
\end{equation}
$(\ref{7.7})$ and $(\ref{7.11})$ imply our assertion.
\qed
\end{Proof}

\begin{Definition}
Let $f:\mathbb{Z}_{\geq 0} \rightarrow \mathbb{C}$ and 
$n\in \mathbb{N}$. 
Let $\left(\tilde{H}^{(k)}f \right) (n)$ be defined by
\[
\left( \tilde{H}^{(k)}f \right)(n)
 = -a_k(n)f({n+1})+d_k(n)f(n)-a_k({n-1})f({n-1}),
\]
where $a_k(0)=1$.
\end{Definition}
Let $z \in \mathbb{C}^+$, and $n,m \in \mathbb{N}$ such that  $n\geq m$. We define 
\begin{eqnarray}
T_z(n)
&=&
\begin{cases}
\left(
\begin{array}{cc}
0&1\\
-\sqrt{\frac{g_{n-1}}{g_n}}&\frac{g_n+1-z}{\sqrt{g_n}}\\
\end{array}
\right)
&
(n \geq 1),
\\
\left(
\begin{array}{cc}
0&1\\
-1&1-z\\
\end{array}
\right)
&(n=0),
\end{cases}
\nonumber
\\
S_z(n,m)&=& T_z(n)T_z(n-1)\cdots T_z(m),
\nonumber
\\
S_z(n)&=& S_z(n,0).
\nonumber 
\end{eqnarray}
%\begin{Lemma}\label{7.1.3}
Let $f : \mathbb{Z}_{\geq 0}\rightarrow \mathbb{C}$ and $z\in \mathbb{C}^+$.
Suppose that $\left( \tilde{H}^{(k)}f \right) (n)=zf(n)$ for each $n \in \mathbb{N}$. Then 
\begin{eqnarray}
\left(
\begin{array}{c}
f(n)\\
f(n+1)
\end{array}
\right)
&=&
T_z(n+N(k)-1)
\left(
\begin{array}{c}
f(n-1)\\
f(n)
\end{array}
\right)
\nonumber
\\
&=&
S_z(n+N(k)-1, N(k))
\left(
\begin{array}{c}
f(0)\\
f(1)
\end{array}
\right).
\nonumber
\end{eqnarray}
%\end{Lemma}
\begin{comment}
\begin{Proof}
\rm
%$f:\mathbb{Z}_{\geq 0}\rightarrow \mathbb{C}$, $z \in \mathbb{C}$ に対して
%$\tilde{H}_0f(n)=zf(n)$ $(n \in \mathbb{N})$とする. 
%このとき任意の
We see that 
\begin{eqnarray}
\left(
\begin{array}{c}
f(n)\\
f(n+1)
\end{array}
\right)=
\left(
\begin{array}{cc}
0&1\\
-\frac{a_k(n-1)}{a_k(n)}&\frac{d_k(n)-z}{a(n)}\\
\end{array}
\right)
\left(
\begin{array}{cc}
f(n-1)\\
f(n)
\end{array}
\right).
\nonumber
\end{eqnarray}
By Definition $\ref{Jacobi}$, this implies our assertion.
\qed
\end{Proof}
\end{comment}
\begin{Lemma}\label{7.1.4}
Let $K>0$ and $z=E+i \epsilon$ with $0<E<4$ and $\epsilon>0$. 
Then there exists $C_4=C_4(E, K)>0$ such that 
\begin{enumerate}[$(1)$]
\item
if $L_m +1\leq n < L_{m+1}$ and $n \epsilon<K$, then
\begin{equation}
\|S_z(n)^{-1}\|\leq C_4^{m+1}\prod_{j=1}^{m}L_j^{\frac{1-\Gamma}{2\Gamma}},
\nonumber
\end{equation}
\item
if $n \leq L_{m}$ and $n \epsilon<K$, then
\begin{equation}\label{7.15}
\|S_z(n)^{-1}\|\leq C_4^{m}\prod_{j=1}^{m-1}L_j^{\frac{1-\Gamma}{2\Gamma}}.
\nonumber
\end{equation}
\end{enumerate}
\end{Lemma}

\begin{Proof}\rm
By Definition $\ref{sparse}$, we see that
\begin{equation}\label{7.16}
T_z(n)=
\begin{cases}
\left(
\begin{array}{cc}
0&1\\
-[n^{\frac{1-\Gamma}{\Gamma}}]^{-\frac{1}{2}}&
\frac{[n^{\frac{1-\Gamma}{\Gamma}}]+1-z}{[n^{\frac{1-\Gamma}{\Gamma}}]^{\frac{1}{2}}}\\
\end{array}
\right)
&(n\in \{L_m \mid m\in \mathbb{N}\}),\\
\left(
\begin{array}{cc}
0&1\\
-[n^{\frac{1-\Gamma}{\Gamma}}]^{\frac{1}{2}}&2-z\\
\end{array}
\right)
&(n\in \{L_m+1 \mid m\in \mathbb{N}\}),\\
\left(
\begin{array}{cc}
0&1\\
-1&2-z\\
\end{array}
\right)
&({\rm otherwise}).
\end{cases}
\nonumber
\end{equation}
If $L_m+1 \leq n < L_{m+1}$, then
\begin{eqnarray}
S_z(n)
&=&
R^{n-L_m-1}
S(L_m+1)\nonumber
\\
&=&
R^{n-L_m-1}
S(L_m+1,L_m)
S(L_m-1)
\nonumber
\\
&=&
R^{n-L_m-1}
S(L_m+1,L_m)
R^{L_m-L_{m-1}-2}
S(L_{m-1}+1)
\nonumber
\\
&=&\cdots
\nonumber
\\
&=&
R^{n-L_m-1}
S(L_m+1,L_m)
R^{L_m-L_{m-1}-2}
\cdots
S(L_1+1,L_1)
R^{2}
\nonumber
\label{7.17}
\nonumber
\end{eqnarray}
where 
$\begin{displaystyle}
R
=
\left(
\begin{array}{cc}
0&1\\
-1&2-z\\
\end{array}
\right).
\end{displaystyle}$
Let $L_0=-2$. Then we see that
\begin{equation}
\|S_z(n)^{-1}\|
\leq
\|R^{-n+L_m+1}\|
\|T_z(0)\|
\prod_{j=1}^{m}
\|S(L_j+1,L_j)^{-1}\|
\|R^{-L_j+L_{j-1}+2}\|.
\label{7.18}
\end{equation}
Note that 
$\begin{displaystyle}
R^{-1}
=
\left(
\begin{array}{cc}
2-z&-1\\
1&0\\
\end{array}
\right) 
\end{displaystyle}$
and $\|R^m\|=\|R^{-m}\|$ for any $m\in \mathbb{N}$.
\begin{eqnarray}
R=
\left(
\begin{array}{cc}
0&1\\
-1&2-E\\
\end{array}
\right)
-i
\left(
\begin{array}{cc}
0&0\\
0&\epsilon\\
\end{array}
\right).
%\label{7.19}
\nonumber
\end{eqnarray}
Since $E\in(0,4)$, there exist invertible matrix $A_E$ and $\lambda_{\pm}\in \mathbb{C}$ with
$|\lambda_{\pm}|=1$
such that 
\begin{equation}
A_E^{-1}
\left(
\begin{array}{cc}
0&1\\
-1&2-E\\
\end{array}
\right)
A_E
=
\left(
\begin{array}{cc}
\lambda_+&0\\
0&\lambda_-\\
\end{array}
\right).
\label{7.20}
\nonumber
\end{equation}
Therefore we see that
\begin{eqnarray}
\|A_E^{-1}R^nA_E\|
&\leq&
\left\|
\left(
\begin{array}{cc}
\lambda_+&0\\
0&\lambda_-\\
\end{array}
\right)
-
i\epsilon
A_E^{-1}
\left(
\begin{array}{cc}
0&0\\
0&1\\
\end{array}
\right)
A_E
\right\|^n
\nonumber
\\
&\leq&
\left(
1+\epsilon\|A_E^{-1}\|\|A_E\|
\right)^n.
\label{7.21}
\nonumber
\end{eqnarray}
If $\epsilon<\frac{K}{n}$, then
\begin{equation}
\|A_E^{-1}R^nA_E\|
\leq
2{\rm exp}(K\|A_E^{-1}\|\|A_E\|).
\nonumber
%\label{7.22}
\end{equation}
Therefore we obtain that for $\epsilon<\frac{K}{n}$,
\begin{eqnarray}
\|R^n\|=\|A_EA_E^{-1}R^nA_EA_E^{-1}\|
\leq 2\|A_E\|\|A_E^{-1}\|{\rm exp}(K\|A_E^{-1}\|\|A_E\|).
\label{SD}
\end{eqnarray}
If $0<\epsilon<1$, then there exists $C_4^{\prime}>0$ such that 
\begin{equation}
\|S_z(L_j+1,L_j)^{-1}\|
\leq
C_4^{\prime}[L_j^{\frac{1-\Gamma}{\Gamma}}]^{\frac{1}{2}}
\leq
C_4^{\prime}L_j^{\frac{1-\Gamma}{2\Gamma}}.
\label{7.23}
\end{equation}
By $(\ref{7.18})$, $(\ref{SD})$ and $(\ref{7.23})$, we see that for $L_m +1\leq n < L_{m+1}$,
\begin{equation}
\|S_z(n)^{-1}\|
\leq
(2C_4^{\prime}\|A_E\|\|A_E^{-1}\|)^{m+1}{\exp}((m+1)K\|A_E^{-1}\|\|A_E\|)
\prod_{j=1}^{m}
L_j^{\frac{1-\Gamma}{2\Gamma}}.
\nonumber
\end{equation}
This implies our first part of the assertion. The second part of the assertion can be proved similarly.
\qed
\end{Proof}

\begin{Lemma}\label{7.1.5}
Let $\psi \in l^2(\mathbb{N})$ and $n \in \mathbb{N}$. 
Then for any $T>0$,
\begin{equation}
\cfrac{1}{T}\int_{0}^{\infty} e^{-\frac{t}{T}}|\psi_k(t,n)|^2 \:dt
=
\cfrac{\epsilon}{\pi}\int_{\mathbb{R}}|(H^{(k)}-(E+i\epsilon))^{-1}\psi(n)|^2\:dE,
\nonumber
\end{equation}
where $\epsilon=\frac{1}{2T}$.
\end{Lemma}
\begin{Proof}\rm
Abbreviate $H^{(k)}$ and $\psi_k$ to $H$ and $\psi$, respectively. We see that 
\begin{eqnarray}
\cfrac{1}{T}\int_{0}^{\infty} e^{-\frac{t}{T}}|\psi(t,n)|^2 \:dt
&=&
\cfrac{1}{T}\int_{0}^{\infty} dt e^{-\frac{t}{T}}(\delta_n, e^{-itH}\psi)(e^{-itH}\psi,\delta_n)
\nonumber\\
&=& 
\int_{\mathbb{R}}(\delta_n, E(dx)\psi)
\int_{\mathbb{R}}(E(dy)\psi,\delta_n)
\cfrac{1}{T}\int_{0}^{\infty} dt e^{-\frac{t}{T}-it(x-y)}
\nonumber\\
&=& 
\int_{\mathbb{R}}(\delta_n, E(dx)\psi)
\int_{\mathbb{R}}(E(dy)\psi,\delta_n)
(1+iT(x-y))^{-1}, 
\nonumber
\end{eqnarray}
and
\begin{eqnarray}
&&\cfrac{\epsilon}{\pi}\int_{\mathbb{R}}|(H^{(k)}-(E+i\epsilon))^{-1}\psi(n)|^2\:dE
\nonumber
\\
&&=
\int_{\mathbb{R}}(\delta_n, E(dx)\psi)
\int_{\mathbb{R}}(E(dy)\psi,\delta_n)
\cfrac{\epsilon}{\pi}
\int_{\mathbb{R}}(E-x+i\epsilon)^{-1}(E-y-i\epsilon)^{-1}
\nonumber\\
&&=
\int_{\mathbb{R}}(\delta_n, E(dx)\psi)
\int_{\mathbb{R}}(E(dy)\psi,\delta_n)
(1+iT(x-y))^{-1}.
\nonumber
\end{eqnarray}
These imply our assertion.
\qed
\end{Proof}

\begin{Definition}
Let $f:\mathbb{R} \rightarrow \mathbb{C}$ be measurable and 
$B_{\nu}=[\nu, 4-\nu]$ with $0<\nu<1$.
We say that $f$ is the first kind, if there exist 
$\nu>0$ and $x_0 \in B_\nu$ such that $f \in C^{\infty}_0(B_\nu)$ and $f(x_0)\neq 0$,
and we say that
$f$ is the second kind, if 
$f$ is bounded and there exist $E_0\in (0,4)$ and $\nu>0$ with 
$[E_0-\nu,E_0+\nu]\subset B_\nu$ such that 
$f \in C^{\infty}([E_0-\nu,E_0+\nu])$ and 
$|f(x)|\geq c>0$ for $x\in [E_0-\nu,E_0+\nu]$. 

\end{Definition}

\begin{Lemma}\label{7.1.6}
%$f \in C^{\infty}_0(B_{\nu})$. 
Let $f:\mathbb{R} \rightarrow \mathbb{C}$ be the second kind and $\psi=f(H^{(k)})\delta_1$.
Let $N$ be sufficiently large. Then there exists $C_5=C_5(\nu)>1$ such that 
\begin{enumerate}[$(1)$]
\item
if $L_N\leq T \leq \cfrac{L_{N+1}}{4}$, then
\begin{eqnarray}
P_{\psi}^{(k)}(\{T \sim \infty\}, T) \geq C_5^{-(N+1)}T
\left(
\cfrac{1}{T}+I_{\delta_1}^{(k)}(T^{-1}, B_{\nu})
\right)
\prod_{j=1}^{N}L_j^{\frac{\Gamma-1}{\Gamma}},
\label{620}
\end{eqnarray}
\item
if $\cfrac{L_N}{4}\leq T \leq L_{N}$, then
\begin{eqnarray}
P_{\psi}^{(k)}(\{T \sim \infty\}, T) \geq C_5^{-(N+1)}L_N
\left(
\cfrac{1}{T}+I_{\delta_1}^{(k)}(T^{-1}, B_{\nu})
\right)
\prod_{j=1}^{N}L_j^{\frac{\Gamma-1}{\Gamma}},
\label{621}
\end{eqnarray}
\item
if $\cfrac{L_N}{4} \leq T$, then
\begin{eqnarray}
P_{\psi}^{(k)}(
\{\frac{L_N}{4} \sim \frac{L_N}{2}\}, T) \geq C_5^{-N}L_N
\left(
\cfrac{1}{T}+I_{\delta_1}^{(k)}(T^{-1}, B_\nu)
\right)\prod_{j=1}^{N-1}L_j^{\frac{1-\Gamma}{\Gamma}}.
\label{622}
\end{eqnarray}
\end{enumerate}
\end{Lemma}
\begin{Proof}\rm
Firstly, we prove the lemma in the case of $\psi=\delta_1$. 
Let $z \in \mathbb{C}^+$ and $f_k: \mathbb{Z}_{\geq0} \rightarrow \mathbb{C}$ be 
\begin{equation}
f_k(n)= 
\begin{cases}
(H^{(k)}-z)^{-1}\delta_1(n)&(n \in \mathbb{N}),\\
1&(n=0).
\end{cases}
\nonumber
\end{equation}
Let $g=(H^{(k)}-z)^{-1}\delta_1 \in l^2(\mathbb{N})$. We see that $g(n)=f_k(n)$ for each $n \in \mathbb{N}$ and that 
\begin{eqnarray}
&&
(H^{(k)}-z)g(n)=\delta_1(n)
\nonumber\\
&&
\Leftrightarrow
\begin{cases}
-a_k(n)g(n+1)+d_k(n)g(n)-a_k(n-1)g(n-1)-zg(n)=0&(n\geq2)\\
-a_k(1)g(2)+d_k(1)g(1)-zg(1)=1&(n=1)
\end{cases}
\nonumber\\
&&
\Leftrightarrow
\begin{cases}
-a_k(n)f_k(n+1)+d_k(n)f_k(n)-a_k(n-1)f_k(n-1)=zf_k(n)&(n\geq2)\\
-a_k(1)f_k(2)+d_k(1)f_k(1)-1=zf_k(1)&(n=1)
\end{cases}
\nonumber\\
&&
\Leftrightarrow
\begin{cases}
-a_k(n)f_k(n+1)+d_k(n)f_k(n)-a_k(n-1)f_k(n-1)=zf_k(n)&(n\geq2)\\
-a_k(1)f_k(2)+d_k(1)f_k(1)-a_k(0)f_k(0)=zf_k(1)&(n=1)
\end{cases}.
\nonumber
\end{eqnarray}
This implies that $f_k$ satisfies the equation $\left(\tilde{H}^{(k)}f_k\right)(n)=z f_k(n)$ for each $n\in \mathbb{N}$. 
We see that 
\begin{eqnarray}
\left(
\begin{array}{c}
f_k(n)\\
f_k(n+1)
\end{array}
\right)=
S_z(n+N(k)-1,N(k))
\left(
\begin{array}{c}
f_k(0)\\
f_k(1)
\end{array}
\right).
\label{7.29}
\end{eqnarray}
Note that $f_k(1)=m_{\delta_1}^{(k)}(z)$. Let $z=E+i\epsilon$.
By $(\ref{7.29})$, we obtain that
\begin{equation}
|f_k(n)|^2+|f_k(n+1)|^2
\geq
\cfrac{1+|m_{\delta_1}^{(k)}(E+i\epsilon)|^2}{\|S_z(n+N(k)-1,N(k))^{-1}\|^2}
.
\nonumber
\end{equation}
Suppose that $L_N+1\leq n+N(k) \leq L_{N+1}$ and $\epsilon<\cfrac{K}{n}$. 
By Lemma $\ref{7.1.4}$, we see that 
\begin{equation}
|f_k(n)|^2+|f_k(n+1)|^2
\geq
C_4^{-(N+1)}\prod_{j=1}^{N}L_j^{\frac{\Gamma-1}{\Gamma}}
\left(
1+|{\rm Im}\;m_{\delta_1}^{(k)}(E+i\epsilon)|^2
\right).
\label{SDF}
\end{equation}
By Lemma $\ref{7.1.5}$, we see that $f_k(n)=(H^{(k)}-z)^{-1}\delta_1(n)$ and that  
\[
\cfrac{\epsilon}{\pi}\int_{\mathbb{R}}|(H^{(k)}-(E+i\epsilon))^{-1}\delta_1(n)|^2\:dE
=\cfrac{1}{T}\int_{0}^{\infty} e^{-\frac{t}{T}}|{\delta_1}_k(t,n)|^2 \:dt=
a_{\delta_1}^{(k)}(n,T),
\qquad \epsilon=(2T)^{-1}.
\]
Let $(2T)^{-1}<\cfrac{K}{n}$. By $(\ref{SDF})$,  we see that there exist 
$C_4$ and $C_4^{\prime}>0$ such that
\begin{eqnarray}
a_{\delta_1}^{(k)}(n,T)+
a_{\delta_1}^{(k)}(n+1,T)
&\geq&
C_4^{-(N+1)}\prod_{j=1}^{N}L_j^{\frac{\Gamma-1}{\Gamma}}
\cfrac{1}{2T}
\int_{B_{\nu}}
dE
\left(
1+|{\rm Im}\;m_{\delta_1}^{(k)}(E+i(2T)^{-1})|^2
\right)
\nonumber
\\
&\geq&
C_4^{-(N+1)}\prod_{j=1}^{N}L_j^{\frac{\Gamma-1}{\Gamma}}
\left(
\cfrac{1}{2T}
+
I_{\delta_1}^{(k)}((2T)^{-1},B_{\nu})
\right)
\nonumber
\\
&\geq&
C_4^{\prime-(N+1)}\prod_{j=1}^{N}L_j^{\frac{\Gamma-1}{\Gamma}}
\left(
\cfrac{1}{T}
+
I_{\delta_1}^{(k)}(T^{-1},B_{\nu})
\right).
\nonumber
%\label{7.32}
\end{eqnarray}
Let $L_N<T<\cfrac{L_{N+1}}{4}$, and $K$ sufficiently large. Then we have
%で$(\ref{7.32})$が成り立つ.故に
\begin{equation}
\sum_{T\leq n \leq 2T}
a_{\psi}^{(k)}(n,T)+
a_{\psi}^{(k)}(n+1,T)
\geq
C_4^{\prime-(N+1)}\prod_{j=1}^{N}L_j^{\frac{\Gamma-1}{\Gamma}}T
\left(
\cfrac{1}{T}
+
I_{\delta_1}^{(k)}(T^{-1},B_{\nu})
\right),
\nonumber
\end{equation}
and hence 
\begin{equation}
P_{\psi}^{(k)}(\{T\sim 2T\},T)
\geq
C_4^{\prime-(N+1)}\prod_{j=1}^{N}L_j^{\frac{\Gamma-1}{\Gamma}}T
\left(
\cfrac{1}{T}
+
I_{\delta_1}^{(k)}(T^{-1},B_{\nu})
\right).
\label{7.33}
\end{equation}
Let $\cfrac{L_N}{4}<T< L_N$. Then we see that
%で$(\ref{7.32})$が成り立つ. 故に
\begin{equation}
P_{\psi}^{(k)}(\{2L_N \sim 3L_N\},T)
\geq
C_4^{\prime-(N+1)}\prod_{j=1}^{N}L_j^{\frac{\Gamma-1}{\Gamma}}L_N
\left(
\cfrac{1}{T}
+
I_{\delta_1}^{(k)}(T^{-1},B_{\nu})
\right).
\nonumber
\end{equation}
Let $\cfrac{L_N}{4} \leq T$. Then we see that 
\begin{equation}
P_{\psi}^{(k)}(\{\cfrac{L_N}{4}\sim \cfrac{L_N}{2}\},T)
\geq
C_4^{\prime-N}\prod_{j=1}^{N-1}L_j^{\frac{\Gamma-1}{\Gamma}}
L_N
\left(
\cfrac{1}{T}
+
I_{\delta_1}^{(k)}(T^{-1},B_{\nu})
\right).
\nonumber
%\label{7.35}
\end{equation}
Therefore, we can prove the lemma in the case of $\psi=\delta_1$.

Next we take $g\in C^{\infty}_0([0,4])$ such that $g(x)=1$ on $B_{\frac{\nu}{2}}$.
We prove the lemma in the case of $\psi=g(H^{(k)})\delta_1$.
Let $\chi = \delta_1-\psi$ and $z \in \mathbb{C}^+$. Then
\begin{eqnarray}
|(H^{(k)}-z)^{-1}\psi(n)|^2
\geq
\cfrac{1}{2}|(H^{(k)}-z)^{-1}\delta_1(n)|^2
-|(H^{(k)}-z)^{-1}\chi(n)|^2.
\nonumber
\end{eqnarray}
%\begin{eqnarray}
%a_{\psi}^{(k)}(n,T)
%\geq
%\cfrac{1}{2}
%a_{\delta_1}^{(k)}(n,T)-a_{\chi}^{(k)}(n,T).
%\end{eqnarray}
Let $L_N<T<\cfrac{L_{N+1}}{4}$. Then 
we see that, by $(\ref{7.33})$, 
\begin{eqnarray}
P_{_\psi}^{(k)}(\{T\sim \infty \},T)
&\geq&
\cfrac{1}{2}\:
C_4^{\prime-(N+1)}\prod_{j=1}^{N}L_j^{\frac{\Gamma-1}{\Gamma}}T
\left(
\cfrac{1}{T}
+
I_{\delta_1}^{(k)}(T^{-1},B_{\nu})
\right)
\nonumber
\\
&-&
\cfrac{1}{T}
\int_{B_{\nu}}dE
\sum_{T\leq n \leq 2T}
|(H^{(k)}-(E+i\epsilon))^{-1}\chi(n)|^2.
\label{FGH}
\end{eqnarray}
Let $f_z(x)= \cfrac{1-g(x)}{x-z}$. Then $(H^{(k)}-z)^{-1}\chi(n)=f_z(H^{(k)})\delta_1(n)$ and 
Lemma $\ref{6.2.4}$ implies that for $l>1$,
\begin{eqnarray}
|(H^{(k)}-z)^{-1}\chi(n)|=|f_z(H^{(k)})\delta_1(n)|&\leq& C_2\opnorm{f_z}_{2l+3} n^{-l},
\nonumber
\end{eqnarray}
and that 
\begin{eqnarray}
\sum_{T\leq n \leq 2T}|(H^{(k)}-z)^{-1}\chi(n)|=\sum_{T\leq n \leq 2T}|f_z(H^{(k)})\delta_1(n)|&\leq& C_2\opnorm{f_z}_{2l+3} T^{-(l-1)}.
\label{FGH1}
\end{eqnarray}
%\begin{equation}
%\sum_{L_m/4\leq n \leq L_m/2}|f_z(H^{(k)})\delta_1(n)|\leq
%C(l)\opnorm{f_z}_{2l+2} L_m^{-(l-1)}.
%\nonumber
%\end{equation}
Let $z=E+i\epsilon$.  Note that there exists $C_5^{\prime}=C_5^{\prime}(g,\nu, l)>0$ such that 
$\begin{displaystyle}
\sup_{E\in B_{\nu}, 0<\epsilon <1}\opnorm{f_z}_{2l+3} \leq C_5^{\prime}.
\end{displaystyle}$
By $(\ref{FGH})$ and $(\ref{FGH1})$, we obtain
\begin{eqnarray}
P_{_\psi}^{(k)}(n \geq T,T)
&\geq&
\cfrac{1}{2}\:
C_4^{\prime-(N+1)}\prod_{j=1}^{N}L_j^{\frac{\Gamma-1}{\Gamma}}T
\left(
\cfrac{1}{T}
+
I_{\delta_1}^{(k)}(T^{-1},B_{\nu})
\right)
-
\cfrac{4}{T}
\:
%\int_{B_{\nu}}dE
C_2C_5^{\prime} T^{-(l-1)}
\nonumber
\\
&\geq&
\left\{
\cfrac{1}{2}\:
C_4^{\prime-(N+1)}
-
4C_2 C_5^{\prime}T^{-l}
\prod_{j=1}^{N}L_j^{\frac{1-\Gamma}{\Gamma}}
\right\}
\prod_{j=1}^{N}L_j^{\frac{\Gamma-1}{\Gamma}}T
\left(
\cfrac{1}{T}
+
I_{\delta_1}^{(k)}(T^{-1},B_{\nu})
\right).
\label{7.37}
\nonumber
\end{eqnarray}
%$(\ref{7.37})$より
Let $l$ sufficiently large, then we can prove $(\ref{620})$ 
in the case of $\psi=g(H^{(k)})$.
We can prove $(\ref{621})$ and $(\ref{622})$ in the case of $\psi=g(H^{(k)})$ similarly.

Finally, let $f$ be the second kind, and we prove the lemma in the case of $\psi=f(H^{(k)})$. 
Let $\nu$ satisfy $f \in C^{\infty}([E_0-\nu,E_0+\nu])$ and 
$|f(x)|\geq c>0$ for $x\in [E_0-\nu,E_0+\nu]$.
We take $g \in C^{\infty}_0([E_0-\nu,E_0+\nu])$ such that 
$g(x)=1$ on $[E_0-\frac{3\nu}{4},E_0+\frac{3\nu}{4}]$. 
Then there exists $h\in C^{\infty}_0([E_0-\nu,E_0+\nu])$ such that $g(x)=h(x)f(x)$.
Since $\sup_n \sum_{m=1}^{\infty}
\langle n-m \rangle^{-l}<\infty$, by Lemma $\ref{6.2.4}$, we see that
\begin{eqnarray}
&&|(H^{(k)}-z)^{-1}g(H^{(k)})\delta_1(n)|^2
\nonumber
\\
&&=
|(h(H^{(k)})\delta_n, (H^{(k)}-z)^{-1}f(H^{(k)})\delta_1 )|^2
\nonumber
\\
&&=
\left|
\sum_{m=1}^{\infty}
(h(H^{(k)})\delta_n,\delta_m)(\delta_m,(H^{(k)}-z)^{-1}f(H^{(k)})\delta_1)
\right|^2
\nonumber
\\
&&
\leq
C_2
\left|
\sum_{m=1}^{\infty}
\langle n-m \rangle^{-l}|(\delta_m,(H^{(k)}-z)^{-1}f(H^{(k)})\delta_1)|
\right|^2
\nonumber
\\
&&
\leq
C_2
\left(
\sum_{m=1}^{\infty}
\langle n-m \rangle^{-l}
\right)
\left(
\sum_{m=1}^{\infty}
\langle n-m \rangle^{-l}
|(H^{(k)}-z)^{-1}f(H^{(k)})\delta_1(m)|^2
\right)
\nonumber
\\
&&
\leq
C_2^{\prime}
\sum_{m=1}^{\infty}
\langle n-m \rangle^{-l}
|(H^{(k)}-z)^{-1}f(H^{(k)})\delta_1(m)|^2.
\nonumber
\end{eqnarray}
This implies the inequality
\begin{eqnarray}
A(2L,T)&\coloneqq&
\epsilon
\sum_{n\geq 2L}
\int_{B_\nu}dE\:
|(H^{(k)}-z)^{-1}g(H^{(k)})\delta_1(n)|^2
\nonumber
\\
&=&
C_2^{\prime}\epsilon
\sum_{m=1}^{\infty}\sum_{n\geq 2L}
\langle n-m \rangle^{-l}
\int_{B_\nu}dE\:
|(H^{(k)}-z)^{-1}f(H^{(k)})\delta_1(m)|^2
\nonumber
\\
&\leq&
\epsilon
\sum_{m=1}^{\infty}
h_l(m,L)
\int_{B_\nu}dE\:
|(H^{(k)}-z)^{-1}f(H^{(k)})\delta_1(m)|^2,
\label{BGT}
\end{eqnarray}
where $z=E+i\epsilon$, $\epsilon=\cfrac{1}{2T}$ and 
$\begin{displaystyle}
h_l(m, L)=
\sum_{n\geq 2L}\cfrac{C_2^{\prime}}{1+ |n-m|^l}.
\end{displaystyle}$
It follows for $\phi \in l^2(\mathbb{N})$ and $\epsilon>0$ that 
\begin{eqnarray}
\epsilon\sum_{n=1}^{\infty}
\int_{\mathbb{R}}dE
|(H^{(k)}-z)^{-1}\phi(n)|^2
=
\pi
\|\phi\|^2,
\qquad z=E+i\epsilon.
\nonumber
\end{eqnarray}
There exists 
$C_2^{\prime \prime}>\max\{C_2^\prime, \sup_{m\geq T}h_l(m,T) \}$. 
By $(\ref{BGT})$, we obtain that 
\begin{eqnarray}
A(2T,T)
&\leq&
\epsilon
\sum_{m<T}^{}
h_l(m,T)
\int_{B_\nu}dE\:
|(H^{(k)}-z)^{-1}f(H^{(k)})\delta_1(m)|^2
\nonumber
\\
&+&
\epsilon
\sum_{m\geq T}^{}
h_l(m,T)
\int_{B_\nu}dE\:
|(H^{(k)}-z)^{-1}f(H^{(k)})\delta_1(m)|^2
\nonumber
\\
&\leq& \pi \|f(H^{(k)})\delta_1\|^2C^{\prime \prime}_2T^{1-l}+
C_2^{\prime\prime}\epsilon\sum_{m \geq T}
\int_{B_\nu}dE\:
|(H^{(k)}-z)^{-1}f(H^{(k)})\delta_1(m)|^2
\nonumber
\\
&=&
\pi \|f(H^{(k)})\delta_1\|^2C_2^{\prime\prime}T^{1-l}+
{C_2^{\prime \prime}}
P_{\psi}^{(k)}(\{T\sim \infty \} , T),
\label{7.38}
\end{eqnarray}
where $z=E+i\epsilon$ and $\epsilon=\cfrac{1}{2T}$.
Let $L_N<T<\cfrac{L_{N+1}}{4}$. Then the previous argument shows that
\begin{eqnarray}
\mbox{

$\begin{displaystyle}
A(2T,T)=\pi P_{g(H^{(k)})\delta_1}^{(k)}(\{2T \sim \infty \}, T)\geq
C_5^{-(N+1)}\prod_{j=1}^{N}L_j^{\frac{\Gamma-1}{\Gamma}}T
\left(
\cfrac{1}{T}
+
I_{\delta_1}^{(k)}(T^{-1},B_{\nu})
\right).
\end{displaystyle}$
}
\label{7.39}
\end{eqnarray}
We take $l$ sufficiently large. Then $(\ref{7.38})$ and $(\ref{7.39})$ imply 
$(\ref{620})$ in the case of $\psi=f(H^{(k)})$. 
We can also prove $(\ref{621})$ and $(\ref{622})$ similarly.
\qed
\end{Proof}
%%%%%%%%%%%%%%%%%%%%%%%%%%%%%%%%%%%%%%%%%%%%%%%%%%%%%%%%%%%%%%%%%%%%%%%
%\begin{Lemma}
%Let $N$ be sifficiently large, then there exists $q_N\in \mathbb{R}$ such that 
%$\begin{displaystyle}
%\lim_{N \rightarrow \infty}q_N=0
%\end{displaystyle}$ and it follows for $\cfrac{L_N}{4}<T<\cfrac{L_{N+1}}{4}$ that
%\begin{eqnarray}
%J_{\delta_1}^{(k)}(T^{-1}, B_\nu)\leq 
%CI_{\delta_1}^{(k)}(T^{-1},B_\nu)\leq 
%C\cfrac{{L_N}^{q_N}}{L_N+TL_N^{\frac{\alpha-1}{\alpha}}}.
%\end{eqnarray}
%\end{Lemma}
%\begin{Proof}\rm
%Let $q_N$ satisfy the equation $L_N^{q_N}=C^{N+1}\prod_{j=1}^{N-1}L_j^{\frac{1-\alpha}%{\alpha}}$, then Lemma $\ref{pro612}$ and Lemma $\ref{7.1.6}$ show this.
%\qed\end{Proof}

\begin{Lemma}\label{618}
Let $f$ be the first kind with $\sup_{x}|f(x)|\leq 1$ and $\psi = f(H^{(k)})\delta_1$. 
If $N$ is sufficient large, then
there exist $C_6>0$ and $q_N\in \mathbb{R}$ such that 
$\begin{displaystyle}
\lim_{N \rightarrow \infty}q_N=0
\end{displaystyle}$ and it follows for $\cfrac{L_N}{4}<T<\cfrac{L_{N+1}}{4}$ that 
\begin{eqnarray}
\langle |X|^p \rangle_{\psi}^{(k)}(T)
&\geq&
C_6I_{\delta_1}^{(k)}(T^{-1}, B_\nu)^{-p}
\nonumber\\
&+&
C_6
\left(
L_N^{p+1+q_N}
+T^{p+1}L_N^{\frac{\Gamma-1}{\Gamma}+q_N}
\right)
I_{\delta_1}^{(k)}(T^{-1},B_\nu).
\nonumber
%\label{QWE}
\end{eqnarray}
\end{Lemma}
\begin{Proof}\rm
Let $M\in \mathbb{N}$. Then it follows that  
$\langle |X|^p \rangle_{\psi}^{(k)}(T)
\geq
M^pP_{\psi}^{(k)}(\{M \sim \infty \},T)$. 
Lemma $\ref{6.1.1}$ implies that 
\[
\langle |X|^p \rangle_{\psi}^{(k)}(T)
\geq
M_T^pP_{\psi}^{(k)}(\{M_T \sim \infty \},T)
\geq C_6^{\prime}J_{\psi}^{(k)}(T^{-1},B_{\nu})^{-p}
\geq C_6^{\prime}J_{\delta_1}^{(k)}(T^{-1},B_{\nu})^{-p}.
\]
By $(\ref{RDX})$, we have 
\begin{equation}
\langle |X|^p \rangle_{\psi}^{(k)}(T)\geq C_6^{\prime\prime}I_{\delta_1}^{(k)}(T^{-1},B_{\nu})^{-p}.
\label{WER}
\end{equation}
Note that $f$ is the second kind. 
For $\frac{L_N}{4}\leq T \leq L_N$, by $(\ref{621})$ we have
\begin{eqnarray}
\langle |X|^p \rangle_{\psi}^{(k)}(T)&\geq& 
T^p
P_{\psi}^{(k)}(\{T \sim \infty\}, T)
\nonumber\\
&\geq&
4C_5^{-(N+1)}T^{p+1}
I_{\delta_1}^{(k)}(T^{-1}, B_{\nu})
\prod_{j=1}^{N}L_j^{\frac{\Gamma-1}{\Gamma}}.
\nonumber
\end{eqnarray}
Let $q_N>0 $ satisfy 
\[
L_N^{q_N}=C_5^{-(N+1)}\prod_{j=1}^{N-1}L_j^{\frac{\Gamma-1}{\Gamma}}.
\]
Then $\lim_{N \rightarrow \infty}q_N=0$ and we see that for $\frac{L_N}{4}\leq T \leq L_N$, 
\begin{eqnarray}
\langle |X|^p \rangle_{\psi}^{(k)}(T)\geq
4T^{p+1}
L_N^{\frac{\Gamma-1}{\Gamma}+q_N}
I_{\delta_1}^{(k)}(T^{-1}, B_{\nu}).
\label{REW}
\end{eqnarray}
For $L_N\leq T \leq \cfrac{L_{N+1}}{4}$, by $(\ref{620})$ we have
\begin{eqnarray}
\langle |X|^p \rangle_{\psi}^{(k)}(T)&\geq&
T^p
P_{\psi}^{(k)}(\{T \sim \infty\}, T)
\nonumber\\
&\geq&
C_5^{-(N+1)}T^{p+1}
I_{\delta_1}^{(k)}(T^{-1}, B_{\nu})
\prod_{j=1}^{N}L_j^{\frac{\Gamma-1}{\Gamma}}
\nonumber\\
&\geq&
T^{p+1}
L_N^{\frac{\Gamma-1}{\Gamma}+q_N}
I_{\delta_1}^{(k)}(T^{-1}, B_{\nu}).
\label{XXXX}
\end{eqnarray}
By $(\ref{REW})$ and $(\ref{XXXX})$, we see that for 
$\cfrac{L_N}{4}\leq T \leq \cfrac{L_{N+1}}{4}$, 
\begin{eqnarray}
\langle |X|^p \rangle_{\psi}^{(k)}(T)\geq
T^{p+1}
L_N^{\frac{\Gamma-1}{\Gamma}+q_N}
I_{\delta_1}^{(k)}(T^{-1}, B_{\nu}).
\label{HFS}
\end{eqnarray}
For $\cfrac{L_N}{4}\leq T$, by $(\ref{622})$ we see that
\begin{eqnarray}
\langle |X|^p \rangle_{\psi}^{(k)}(T)&\geq&
T^p
P_{\psi}^{(k)}(\{T \sim \infty\}, T)
\nonumber\\
&\geq&
C_5^{-N}L_N^{p+1}I_{\delta_1}^{(k)}(T^{-1}, B_{\nu})
\prod_{j=1}^{N-1}L_j^{\frac{\Gamma-1}{\Gamma}}
\nonumber\\
&\geq&
L_N^{p+1+q_N}I_{\delta_1}^{(k)}(T^{-1}, B_{\nu}).
\label{IJN}
\end{eqnarray}
$(\ref{WER})$, $(\ref{HFS})$, and $(\ref{IJN})$ imply our assertion.
\qed
\end{Proof}
\begin{Lemma}
Let $f$ be the first kind with $\sup_{x}|f(x)|\leq 1$ and $\psi = f(H^{(k)})\delta_1$. Then
\begin{eqnarray}
\beta_{\psi}^{(k)}(p)\geq \cfrac{p+1}{p+\frac{1}{\Gamma}}.
\label{WWWWW}
\end{eqnarray}
\end{Lemma}
\begin{Proof}\rm
By Lemma $\ref{618}$, for $x=I_{\delta_1}^{(k)}(T^{-1},B_\nu)$, 
we obtain that 
\[
\langle |X|^p \rangle_{\psi}^{(k)}(T)
\geq
C_6x^{-p}+
C_6
\left(
L_N^{p+1-q_N}
+T^{p+1}L_N^{\frac{\Gamma-1}{\Gamma}-q_N}
\right)x.
\]
Let $f(x)=x^{-p}+Kx$. Then  
$\begin{displaystyle}
\inf_{x>0}f(x)=
c(p)K^{\frac{p}{p+1}}
\end{displaystyle}$, where $c(p)=p^{-\frac{p}{p+1}}+p^{\frac{1}{p+1}}$.
Let $\cfrac{L_N}{4} \leq T \leq \cfrac{L_{N+1}}{4}$. Then there exists $C_6^{\prime}=C_6^{\prime}(p)>0$ such that
\begin{eqnarray}
\langle |X|^p \rangle_{\psi}^{(k)}(T)
&\geq&
c(p)C_6
\left(
L_N^{p+1-q_N}
+T^{p+1}L_N^{\frac{\Gamma-1}{\Gamma}-q_N}
\right)^{\frac{p}{p+1}}
\nonumber
\\
&\geq&
C_6^{\prime}
L_N^{-\frac{p}{p+1}q_N}
\left(L_N^{p+1}+T^{p+1}L_N^{\frac{\Gamma-1}{\Gamma}}\right)^\frac{p}{p+1}.
\nonumber
\end{eqnarray}
For $\cfrac{L_N}{4} \leq T \leq L_N^A$ with $A=\cfrac{p+\frac{1}{\Gamma}}{p+1}$, 
we have 
\begin{eqnarray}
\langle |X|^p \rangle_{\psi}^{(k)}(T)
\geq
C_6^{\prime}L_N^{-\frac{p}{p+1}q_N}L_N^p
\geq
C_6^{\prime}L_N^{-\frac{p}{p+1}q_N}T^{\frac{p}{A}}.
\label{UHB}
\end{eqnarray}
For $L_N^A \leq T \leq \cfrac{L_{N+1}}{4}$, we have 
\begin{eqnarray}
\langle |X|^p \rangle_{\psi}^{(k)}(T)
\geq
C_6^{\prime}L_N^{-\frac{p}{p+1}q_N}T^pL_N^{\frac{\Gamma-1}{\Gamma}\frac{p}{p+1}}
\geq
C_6^{\prime}L_N^{-\frac{p}{p+1}q_N}T^{\frac{p}{A}}.
\label{UHB1}
\end{eqnarray}
$(\ref{UHB})$ and $(\ref{UHB1})$ imply that for sufficiently large $T>0$ and any $\epsilon>0$,
\begin{eqnarray}
\langle |X|^p \rangle_{\psi}^{(k)}(T)\geq
C_6^{\prime}T^{\frac{p}{A}-\epsilon}.
\nonumber
\end{eqnarray}
Therefore we obtain that 
\[
\beta_{\psi}^{(k)}(p)=
\cfrac{1}{p}\liminf_{T \rightarrow \infty}\cfrac{{\rm log}\langle |X|^p\rangle_{\psi}^{(k)}(T)}{{\rm log}T}\geq \cfrac{p+1}{p+\frac{1}{\Gamma}}-\cfrac{\epsilon}{p}.
\]
This implies our assertion.
% $(\ref{WWWWW})$. 
\qed
\end{Proof}

\subsection{Upper bound of intermittency function}
\begin{Lemma}
Let $f$ be the first kind, $\psi = f(H^{(k)})\delta_1$, and $p>0$. 
Then there exists $C_7=C_7(p)>0$ such that 
for $L_N \leq T \leq L_N^{\frac{1}{\Gamma}}$ with $N$ sufficiently large, 
\begin{eqnarray}
\sum_{n \geq 2L_N}n^pa_{\psi}^{(k)}(n,T)\leq C_7T^{p+1}L_N^{-\frac{1}{\Gamma}}.
\end{eqnarray}
\end{Lemma}
\begin{Proof}\rm
We have
%$L_N\leq T \leq L_N^A$とする.
\begin{eqnarray}
\sum_{n \geq 2L_m}n^pa_{\psi}^{(k)}(n,T)
=
\sum_{n=2L_m}^{T^3}n^pa_{\psi}^{(k)}(n,T)
+
\sum_{n> T^3}n^pa_{\psi}^{(k)}(n,T).
\nonumber
\end{eqnarray}
Let $G_t(x)= e^{-itx}$. 
Lemma $\ref{6.2.4}$ shows that for any $l>1$, there exists 
$C_7^{(4)}=C_7^{(4)}(l)>0$ such that
\begin{eqnarray}
\sum_{n> T^3}n^p a_{\psi}^{(k)}
&=&
\sum_{n> T^3}n^p\cfrac{1}{T}\int_{\mathbb{R}}dt\; e^{-\frac{t}{T}}
|(\delta_n, G_t(H^{(k)})f(H^{(k)})\delta_1)|^2
\nonumber
\\
&\leq&
\sum_{n> T^3}n^p\cfrac{C_2}{T}\int_{\mathbb{R}}dt\; e^{-\frac{t}{T}}
\opnorm{G_tf}_{2l+3} n^{-l}
\nonumber
\\
&\leq&
\sum_{n> T^3}n^p
\cfrac{C_7^{(4)}}{T}\int_{\mathbb{R}}dt\; e^{-\frac{t}{T}}
\; t^{2l+3}
n^{-l}
\nonumber
\\
&\leq&
C_7^{(4)}
T^{2l+3}
\sum_{n> T^3}n^{p-l}
\nonumber
\\
&\leq&
C_7^{(4)}T^{-l+3p+6}.
\nonumber
\end{eqnarray}
We take $l$ large enough, so it is sufficient to prove 
$\begin{displaystyle}
\sum_{n=2L_N}^{T^3}n^pa_{\psi}^{(k)}(n,T)
\leq
C_7T^{p+1}L_N^{-\frac{1}{\Gamma}}.
\end{displaystyle}$
Since $f$ is the first kind, $f \in C^{\infty}_0(B_{\nu})$. 
By Lemma $\ref{7.1.5}$,  we have
\begin{eqnarray}
\sum_{n=2L_N}^{T^3}n^pa_{\psi}^{(k)}(n,T)
&=&
\sum_{n=2L_N}^{T^3}n^p
\cfrac{\epsilon}{\pi}
\int_{B_{\frac{\nu}{2}}}dE \; |(H^{(k)}-E-i\epsilon)^{-1}\psi(n)|^2
\nonumber
\\
&+&
\sum_{n=2L_N}^{T^3}n^p
\cfrac{\epsilon}{\pi}
\int_{\mathbb{R} \setminus B_{\frac{\nu}{2}}}dE \; |(H^{(k)}-E-i\epsilon)^{-1}\psi(n)|^2,
\nonumber
\qquad \epsilon=(2T)^{-1}.
\end{eqnarray}
Let $\chi_z(x) = (x-z)^{-1}$. Then Lemma $\ref{6.2.4}$ shows that for any $l>0$, there exists 
$C_7^{(3)}=C_7^{(3)}(l)>0$ such that
\begin{eqnarray}
\cfrac{\epsilon}{\pi}
\sum_{n=2L_N}^{T^3}n^p
\int_{-\infty}^{\frac{\nu}{2}}dE \; |(H^{(k)}-E-i\epsilon)\psi(n)|^2
&\leq&
\cfrac{\epsilon}{\pi}
\sum_{n=2L_N}^{T^3}n^p
\int_{-\infty}^{\frac{\nu}{2}}dE \;
|(\delta_n, \chi_{E+i\epsilon}(H^{(k)})f(H^{(k)})\delta_1)|^2
\nonumber
\\
&\leq&
C_2
\cfrac{\epsilon}{\pi}
\sum_{n=2L_N}^{T^3}n^p
\int_{-\infty}^{\frac{\nu}{2}}dE \;
\opnorm{\chi_{E+i\epsilon}f}^{2}_{2l+3}n^{-2l}
\nonumber
\\
&\leq&
C_2
\cfrac{\epsilon}{\pi}
\sum_{n=2L_N}^{T^3}n^{p-2l}
\int_{-\infty}^{\frac{\nu}{2}}dE \;
\cfrac
{C_7^{(3)}}
{
(E-\nu)^2+\epsilon^2
}
\nonumber
\\
&\leq&
C_2C_7^{(3)}L_N^{-2l+p+1}.
\label{FDS}
\end{eqnarray}
Similarly, there exists $D_7^{(3)}=D_7^{(3)}(l)>0$ such that
\begin{eqnarray}
\cfrac{\epsilon}{\pi}
\sum_{n=2L_N}^{T^3}n^p
\int_{4-\frac{\nu}{2}}^{\infty}dE \; |(H^{(k)}-E-i\epsilon)^{-1}\psi(n)|^2
\leq
C_2D_7^{(3)}L_N^{-2l+p+1}.
\label{FDS1}
\end{eqnarray}
$(\ref{FDS})$ and $(\ref{FDS1})$ imply that 
\[
\sum_{n=2L_N}^{T^3}n^p
\cfrac{\epsilon}{\pi}
\int_{\mathbb{R} \setminus B_{\frac{\nu}{2}}}dE \; |(H^{(k)}-E-i\epsilon)^{-1}\psi(n)|^2
\leq
\max\{C_2C_7^{(3)}, C_2D_7^{(3)}\}L_N^{-2l+p+1}.
\]
We take $l$ large enough, so it is sufficient to prove 
\begin{eqnarray}
\cfrac{\epsilon}{\pi}
\sum_{n=2L_N}^{T^3}n^p
\int_{B_{\frac{\nu}{2}}}dE \; |(H^{(k)}-E-i\epsilon)^{-1}\psi(n)|^2
\leq
C_7T^{p+1}L_N^{-\frac{1}{\Gamma}},
\qquad \epsilon=(2T)^{-1}.
\nonumber
\end{eqnarray}
Lemma $\ref{6.2.4}$ implies that there exists $C=C(l,f)>0$ such that
\begin{eqnarray}
|(H^{(k)}-E-i\epsilon)^{-1}\psi(n)|^2
\leq
C \sum_{m=1}^{\infty}
(1+|n-m|^2)^{-l}
|\chi_{E+i\epsilon}(H^{(k)})\delta_1(m)|^2.
\nonumber
\end{eqnarray}
Therefore it is sufficient to prove 
\[
\cfrac{\epsilon}{\pi}
\sum_{n=2L_N}^{T^3}n^p
\int_{B_{\frac{\nu}{2}}}dE \; 
 \sum_{m=1}^{\infty}
(1+|n-m|^2)^{-l}
|\chi_{E+i\epsilon}(H^{(k)})\delta_1(m)|^2
\leq
C_7T^{p+1}L_N^{-\frac{1}{\Gamma}},
\qquad \epsilon=(2T)^{-1}.
\nonumber
\]
We see that for any $l>1$, by Lemma $\ref{7.1.5}$, there exists 
$C_7^{(2)}=C_7^{(2)}(l)>0$ such that 
\begin{eqnarray}
&&
\cfrac{\epsilon}{\pi}
\sum_{n=2L_N}^{T^3}n^p
\int_{B_{\frac{\nu}{2}}}dE \; 
\sum_{m=1}^{L_N}
(1+|n-m|^2)^{-l}
|\chi_{E+i\epsilon}(H^{(k)})\delta_1(m)|^2
\nonumber
\\
&&\leq
\sum_{n=2L_N}^{T^3}n^p
L_N^{-2l}
\sum_{m=1}^{L_N}
\cfrac{\epsilon}{\pi}
\int_{B_{\frac{\nu}{2}}}dE \; 
|\chi_{E+i\epsilon}(H^{(k)})\delta_1(m)|^2
\nonumber
\\
&&\leq
\sum_{n=2L_N}^{T^3}n^p
L_N^{-2l}
\nonumber
\\
&&\leq
C_7^{(2)}
T^{3(p+1)}
L_N^{-2l}, \qquad \epsilon=(2T)^{-1}.
\label{oiu}
\end{eqnarray}
Similarly,  we have
\begin{eqnarray}
\cfrac{\epsilon}{\pi}
\sum_{n=2L_N}^{T^3}n^p
\int_{B_{\frac{\nu}{2}}}dE \; 
\sum_{m=T^3+L_N}^{\infty}
(1+|n-m|^2)^{-l}
|\chi_{E+i\epsilon}(H^{(k)})\delta_1(m)|^2
\leq
D_{7}^{(2)}T^{3(p+1)}
L_N^{-2l}.
\label{oiu1}
\end{eqnarray}
By $(\ref{oiu})$ and $(\ref{oiu1})$,  it is sufficient to prove 
\begin{eqnarray}
\cfrac{\epsilon}{\pi}
\sum_{n=2L_N}^{T^3}n^p
\int_{B_{\frac{\nu}{2}}}dE \; 
\sum_{m=L_N}^{T^3+L_N}
(1+|n-m|^2)^{-l}
|\chi_{E+i\epsilon}(H^{(k)})\delta_1(m)|^2
\leq
C_7T^{p+1}L_N^{-\frac{1}{\Gamma}},\qquad \epsilon=(2T)^{-1}.
\nonumber
\end{eqnarray}
Let $A^{(k)}_N$ and $D^{(k)}_N:l^2(\mathbb{N})\rightarrow l^2(\mathbb{N})$ be 
\begin{displaymath}
A^{(k)}_N=
\renewcommand{\arraystretch}{1.8}
\left(
\begin{tabular}{Wc{10mm}Wc{10mm}Wc{10mm}Wc{10mm}Wc{10mm}|
Wc{10mm}Wc{10mm}Wc{10mm}cccccccccc}
0&$a_{k}({1})$&&&&&\\
$a_{k}({1}$)&0&$a_{k}({2})$&&&&&&\\
&$a_{k}({2})$&$\ddots$&$\ddots$&&&&\\
&&$\ddots$&0&$a_{k}(L_N)$&&&&\\
&&&$a_{k}(L_N)$&0&1\\
\hline
&&&&1&0&1\\
&&&&&1&0&1\\
&&&&&&1&$\ddots$&$\ddots$\\
&&&&&&&$\ddots$&\\
\end{tabular}\right)
\renewcommand{\arraystretch}{1}
\end{displaymath}
and
\begin{displaymath}
D^{(k)}_N=
\renewcommand{\arraystretch}{1.8}
\left(
\begin{tabular}{Wc{8mm}Wc{8mm}Wc{8mm}Wc{10mm}|Wc{8mm}cccccc}
$d_{k}({1})$&&&&\\
&$d_{k}({2})$&&\\
&&$\ddots$&\\
&&&$d_{k}({L_N})$&\\
\hline
&&&&2\\
&&&&&2\\
&&&&&&$\ddots$\\
\end{tabular}\right).
\renewcommand{\arraystretch}{1}
\end{displaymath}
Let $H_N^{(k)}= D^{(k)}_N-A^{(k)}_N$. We see that for any $l>1$ there exists 
$C^{\prime}=C^{\prime}(l)>0$ such that 
\begin{eqnarray}
&&
\cfrac{\epsilon}{\pi}
\sum_{n=2L_N}^{T^3}n^p
\int_{B_{\frac{\nu}{2}}}dE \; 
\sum_{m=L_N}^{T^3+L_N}
(1+|n-m|^2)^{-l}
|\chi_{E+i\epsilon}(H^{(k)})\delta_1(m)|^2
\nonumber
\\
&&=
\cfrac{\epsilon}{\pi}
\sum_{m=L_N}^{T^3+L_N}
\left(
\sum_{n=2L_N}^{T^3}n^p
(1+|n-m|^2)^{-l}
\right)
\int_{B_{\frac{\nu}{2}}}dE \; 
|\chi_{E+i\epsilon}(H^{(k)})\delta_1(m)|^2
\nonumber
\\
&&\leq
C^{\prime}
\cfrac{\epsilon}{\pi}
\sum_{m=L_N}^{T^3+L_N}
m^p
\int_{B_{\frac{\nu}{2}}}dE \; 
|\chi_{E+i\epsilon}(H^{(k)})\delta_1(m)|^2
\nonumber
\\
&&\leq
2C^{\prime}
\cfrac{\epsilon}{\pi}
\sum_{m=L_N}^{T^3+L_N}
m^p
\int_{B_{\frac{\nu}{2}}}dE \; 
|\chi_{E+i\epsilon}(H^{(k)})\delta_1(m)-\chi_{E+i\epsilon}(H^{(k)}_N)\delta_1(m)|^2
\nonumber
\\
&&+
2C^{\prime}
\cfrac{\epsilon}{\pi}
\sum_{m=L_N}^{T^3+L_N}
m^p
\int_{B_{\frac{\nu}{2}}}dE \; 
|\chi_{E+i\epsilon}(H^{(k)}_N)\delta_1(m)|^2.
%\nonumber
\label{6.28}
\end{eqnarray}
By the resolvent equation, we have  
%$\begin{displaystyle}
%(H^{(k)}-H^{(k)}_N)\phi(n)=
%\sum_{m=N+1}^{\infty}d_{k}(m)\phi(m)\delta_m(n)
%\end{displaystyle}$が成り立つ. 
\begin{eqnarray}
&&
\|(\chi_{E+i\epsilon}(H^{(k)})-\chi_{E+i\epsilon}(H^{(k)}_N))\delta_1\|
\nonumber
\\
&&\leq
\cfrac{1}{\epsilon}
\|(H^{(k)}-H_N^{(k)})\chi_{E+i\epsilon}(H^{(k)}_N)\delta_1\|
\nonumber
\\
&&\leq
\cfrac{1}{\epsilon}
\|(D^{(k)}-D_N^{(k)})\chi_{E+i\epsilon}(H^{(k)}_N)\delta_1\|
+
\cfrac{1}{\epsilon}
\|(A^{(k)}-A^{(k)}_N)\chi_{E+i\epsilon}(H^{(k)}_N)\delta_1\|.
\label{6.29}
%\nonumber
\end{eqnarray}
Let $z=E+i\epsilon$ and 
$\phi:\mathbb{Z}_{\geq 0}\rightarrow \mathbb{C}$ be 
$\phi(0)=1$, 
$\phi(n)= \chi_{z}(H^{(k)}_N)\delta_1(n)$ $(n\geq 1)$. Then 
it follows for
$n>L_N$ that 
\begin{eqnarray}
-\phi(n+1)+2\phi(n)-\phi(n-1)=z \phi(n).
\nonumber
\end{eqnarray}
This implies for $n>L_N$,
\begin{eqnarray}
\left(
\begin{array}{c}
\phi(n)\\
\phi(n+1)
\end{array}
\right)
=
\left(
\begin{array}{cc}
0&1\\
-1&2-z\\
\end{array}
\right)
\left(
\begin{array}{c}
\phi(n-1)\\
\phi(n)
\end{array}
\right).
\nonumber
\end{eqnarray}
Let $\lambda_\pm = \cfrac{2-z\pm \sqrt{(2-z)^2+4}}{2}$, then 
there exists $C_\pm$ such that 
\begin{eqnarray}
\phi(n)=C_+ \lambda_+^{n-L_N}+C_-\lambda_-^{n-L_N}.
\nonumber
\end{eqnarray}
Since $\epsilon>0$, we get $|\lambda_-|<1$ and $|\lambda_+|>1$.
Since $\|\phi \|_{l^2}<\infty$, $C_+=0$, $C_-=\phi(L_N)$ and there exists $c,c^{\prime}>0$ such that 
it follows for $0<E<4$ and $0<\epsilon<1$ that
\begin{equation}
e^{-c^{\prime}\epsilon}\leq |\lambda_-|\leq e^{-c\epsilon}.
\nonumber
\end{equation}
If $\beta>1$, then it follows for sufficiently large $N$ that 
\begin{eqnarray}
L_{N+1}-L_N^\beta>\cfrac{1}{2}L_{N+1}.
\nonumber
\end{eqnarray}
Therefore, we see that there exist $C^{\prime\prime}>0$ such that
\begin{eqnarray}
\|(D^{(k)}-D_N^{(k)})\chi_{E+i\epsilon}(H^{(k)}_N)\delta_1\|^2
&=&
\sum_{j=N+1}^{\infty}|d_k(L_j)-2|^2|\phi(L_j)|^2
\nonumber
\\
&\leq&
2|\phi(L_N)|^2
\sum_{j=N+1}^{\infty}L_j^{\frac{2(1-\Gamma)}{\Gamma}}
{\rm exp}(-2c\epsilon(L_j-L_N))
\nonumber
\\
&\leq&
2\epsilon^{-2}
\sum_{j=N+1}^{\infty}L_j^{\frac{2(1-\Gamma)}{\Gamma}}
{\rm exp}(-2c\epsilon(L_j-L_N^{\beta}))
{\rm exp}(-2c\epsilon(L_N^{\beta}-L_N))
\nonumber
\\
&\leq&
4T^2
\sum_{j=N+1}^{\infty}L_j^{\frac{2(1-\Gamma)}{\Gamma}}
{\rm exp}(-c\epsilon L_j)
{\rm exp}(-c\epsilon L_N^{\beta})
\nonumber
\\
&\leq&
C^{\prime\prime}{\rm exp}(-c\epsilon L_N^{\beta}),
\label{6.30}
\\
\nonumber
\end{eqnarray}
and
\begin{eqnarray}
\|(A^{(k)}-A^{(k)}_N)\chi_{E+i\epsilon}(H^{(k)}_N)\delta_1\|^2
&\leq&
\sum_{j=N+1}^{\infty}
|1-a_k(L_j)|^2
\{
|\phi(L_j)|^2+|\phi(L_j+1)|^2
\}
\nonumber
\\
&\leq&
C^{\prime\prime}T^2\sum_{j=N+1}^{\infty}
{\rm exp}(-2c\epsilon(L_j-L_N))
\nonumber
\\
&\leq&
C^{\prime\prime}T^2{\rm exp}(-c\epsilon L_N^{\beta}).
\label{6.31}
\end{eqnarray}
By $(\ref{6.29})$, $(\ref{6.30})$, and $(\ref{6.31})$, we see that there exists 
$C_7^{(1)}>0$ such that
\begin{eqnarray}
&&
\epsilon
\sum_{m=L_N}^{T^3+L_N}
m^p
\int_{B_{\frac{\nu}{2}}}dE \; 
|\chi_{E+i\epsilon}(H^{(k)})\delta_1(m)-\chi_{E+i\epsilon}(H^{(k)}_N)\delta_1(m)|^2
\nonumber
\\
&&\leq
\epsilon
\sum_{m=L_N}^{T^3+L_N}
m^p
\int_{B_{\frac{\nu}{2}}}dE \; 
\|\chi_{E+i\epsilon}(H^{(k)})\delta_1-\chi_{E+i\epsilon}(H^{(k)}_N)\delta_1\|^2
\nonumber
\\
&&\leq
C_7^{(1)}
T^{3(p+1)}{\rm exp}(-c\epsilon L_N^{\beta}),\qquad \epsilon=(2T)^{-1}.
\label{WAZ}
\end{eqnarray}
By $(\ref{6.28})$ and $(\ref{WAZ})$, it is sufficient to prove
\begin{eqnarray}
\epsilon
\sum_{m=L_N}^{T^3+L_N}
m^p
\int_{B_{\frac{\nu}{2}}}dE \; 
|\chi_{E+i\epsilon}(H^{(k)}_N)\delta_1(m)|^2
\leq
C_7T^{p+1}L_N^{-\frac{1}{\Gamma}}.
\nonumber
\end{eqnarray}
Let $F_N(z)= (\delta_1, \chi_z(H_N^{(k)})\delta_1)=(\delta_1,(H_N^{(k)}-z)^{-1}\delta_1)$.
We see that there exists $C^{\prime\prime\prime}>0$ such that
\begin{eqnarray}
\cfrac{1}{\epsilon}{\rm Im}F_N(E+i\epsilon)
=\|\chi_{E+i\epsilon}(H_N^{(k)})\delta_1\|^2
\geq
\sum_{m>L_N}
|\phi(m)|^2
\geq
\sum_{m>L_N}e^{-c^{\prime}\epsilon(m-L_N)}|\phi(L_N)|^2
\geq
\cfrac{C^{\prime\prime\prime}}{\epsilon}|\phi(L_N)|^2.
\nonumber
\end{eqnarray}
This implies that $|\phi(L_N)|\leq C^{\prime\prime\prime} {\rm Im}F_N(E+i\epsilon)$. 
Let $L_{N-1}<n<L_{N+1}$. Then
\[
\left\{
\begin{array}{ll}
-\phi(n+1)+2\phi(n)-\phi(n-1)=z\phi(n)&(n \neq L_N, L_N+1)
\\
-\sqrt{[L_N^{\frac{1-\Gamma}{\Gamma}}]}\phi(L_N+1)+([L_N^{\frac{1-\Gamma}{\Gamma}}]+2)\phi(L_N)-\phi(L_N-1)=z\phi(L_N)&(n= L_N)
\\
-\phi(L_N+2)+2\phi(L_N+1)-\sqrt{[L_N^{\frac{1-\Gamma}{\Gamma}}]}\phi(L_N)=z\phi(L_N+1)&(n =L_N+1).
\end{array}
\right.
\]
Let 
$\begin{displaystyle}
R
=
\left(
\begin{array}{cc}
0&1\\
-1&2-z\\
\end{array}
\right)
\end{displaystyle}$. Then it follows for $L_N+1<n < L_{N+1}$ that 
\begin{eqnarray}
\left(
\begin{array}{c}
\phi(n)\\
\phi(n+1)
\end{array}
\right)
=
R^{n-L_N}
\left(
\begin{array}{c}
\phi(L_N)\\
\phi(L_N+1)
\end{array}
\right).
\nonumber
\end{eqnarray}
Similary, for $L_{N-1} < n<L_N-1$, we have
\begin{eqnarray}
\nonumber
\left(
\begin{array}{c}
\phi(n)\\
\phi(n+1)
\end{array}
\right)
=
R^{n-L_N+1}
\left(
\begin{array}{c}
\phi(L_N-1)\\
\phi(L_N)
\end{array}
\right).
\end{eqnarray}
There exists $B=B(K)>0$ such that $\|R^n\|<B$ for $\epsilon<\frac{K}{|n|}$. 
It follows for $L_N<n<2L_N$ that 
\begin{eqnarray}
|\phi(n)|^2+|\phi(n+1)|^2 \geq B^{-1}(|\phi(L_N)|^2+|\phi(L_N+1)|^2).
\nonumber
\end{eqnarray}
Therefore, we have
\begin{eqnarray}
\cfrac{1}{\epsilon}\:{\rm Im}F_N(E+i\epsilon)
=
\|\phi\|^2 \geq B^{-1}L_N(|\phi(L_N)|^2+|\phi(L_N+1)|^2).
\label{++}
\end{eqnarray}
Similarly, it follows that for $\cfrac{L_N}{2}<n<L_N$, 
\[
|\phi(n)|^2+|\phi(n+1)|^2 \geq B^{-1}(|\phi(L_N)|^2+|\phi(L_N+1)|^2),
\]
and that 
\begin{eqnarray}
\cfrac{1}{\epsilon}{\rm Im}F_N(E+i\epsilon)
=
\|\phi\|^2 \geq B^{-1}L_N(|\phi(L_N-1)|^2+|\phi(L_N)|^2).
\label{+++}
\end{eqnarray}
$(\ref{++})$ and $(\ref{+++})$ imply that
\begin{eqnarray}
|\phi(L_N-1)|^2+|\phi(L_N+1)|^2
\leq 
\cfrac{2B}{\epsilon L_N}{\rm Im}F_N(E+i\epsilon).
\label{IJM}
\end{eqnarray}
We see that
\begin{eqnarray}
([L_N^{\frac{1-\Gamma}{\Gamma}}]+2-z)\phi(L_N)
=\phi(L_N-1)+\sqrt{[L_N^{\frac{1-\Gamma}{\Gamma}}]}\phi(L_N+1).
\nonumber
\end{eqnarray}
This shows that 
\begin{eqnarray}
|[L_N^{\frac{1-\Gamma}{\Gamma}}]+2-z|^2|\phi(L_N)|^2
\leq2[L_N^{\frac{1-\Gamma}{\Gamma}}](|\phi(L_N-1)|^2+|\phi(L_N+1)|^2).
\nonumber
\end{eqnarray}
Let $|z|<5$. Then there exists $B^{\prime}>0$ such that
\begin{eqnarray}
|\phi(L_N)|^2
\leq B^{\prime}L_N^{\frac{\Gamma-1}{\Gamma}}(|\phi(L_N-1)|^2+|\phi(L_N+1)|^2).
\label{IJM1}
\end{eqnarray}
$(\ref{IJM})$ and $(\ref{IJM1})$ imply that
\begin{eqnarray}
|\phi(L_N)|^2
\leq2BB^{\prime}
\cfrac{L_N^{-\frac{1}{\Gamma}}}{\epsilon}{\rm Im}F_N(E+i\epsilon).
\nonumber
\end{eqnarray}
Therefore, there exists $C_7=C_7(p)$ such that 
\begin{eqnarray}
&&
\epsilon
\sum_{m=L_N}^{T^3+L_N}
m^p
\int_{B_{\frac{\nu}{2}}}dE \; 
|\chi_{E+i\epsilon}(H^{(k)}_N)\delta_1(m)|^2
\nonumber
\\
&&\leq
C_7
\epsilon
\sum_{m=L_N}^{T^3+L_N}
m^p
{\rm exp}(-2c\epsilon(m-L_N))
\int_{B_{\frac{\nu}{2}}}dE \; 
|\chi_{E+i\epsilon}(H^{(k)}_N)\delta_1(L_N)|^2
\nonumber
\\
&&\leq
C_7
\epsilon^{-p}
\int_{B_{\frac{\nu}{2}}}dE \; 
|\phi(L_N)|^2
\nonumber
\\
&&\leq
2BB^{\prime}
C_7
\epsilon^{-p-1}
L_N^{-\frac{1}{\Gamma}}
\int_{B_{\frac{\nu}{2}}}dE \; {\rm Im}F_N(E+i\epsilon)
\nonumber
\\
&&\leq
2BB^{\prime}
C_7
T^{p+1}L_N^{-\frac{1}{\Gamma}},\qquad \epsilon=(2T)^{-1}.
\nonumber
\end{eqnarray}
\qed
\end{Proof}
\begin{_corollary}\label{PQ}
Let $p>0$, $f$ be the first kind, and $\psi = f(H^{(k)})\delta_1$. 
Then there exists  $C_8=C_8(p)>0$ such that for 
$L_N \leq T \leq L_N^{\frac{1}{\Gamma}}$ with $N$ sufficiently large, 
\begin{equation}
\langle|X|^p \rangle^{(k)}_{\psi}(T)
\leq
C_8L_N^p+C_8T^{p+1}L_N^{-\frac{1}{\Gamma}}.
\nonumber
\end{equation}
\end{_corollary}
\begin{Lemma}\label{HH}
Let $f$ be the first kind and $\psi = f(H^{(k)})\delta_1$. Then  
\begin{eqnarray}
\beta_{\psi}^{(k)}(p)= \cfrac{p+1}{p+\frac{1}{\Gamma}}.
\nonumber
\end{eqnarray}
\end{Lemma}
\begin{Proof}\rm
Let $L_N \leq T=L_N^A \leq L_N^{\frac{1}{\Gamma}}$, where $A=\cfrac{p+\frac{1}{\Gamma}}{p+1}$. Then Corollary $\ref{PQ}$ shows that 
\begin{eqnarray}
\langle|X|^p \rangle^{(k)}_{\psi}(L_N^A)
\leq
C_8L_N^p.
\nonumber
\end{eqnarray}
Therefore we have 
\begin{eqnarray}
\beta_{\psi}^{(k)}(p)\leq
\cfrac{1}{p}\lim_{N \rightarrow \infty}
\cfrac{{\rm log}\langle |X|^p\rangle_{\psi}^{(k)}(L_N^A)}{{\rm log}L_N^A}
\leq
A^{-1}
=
\cfrac{p+1}{p+\frac{1}{\Gamma}}.
\nonumber
\end{eqnarray}
Since $f$ is the first kind, $(\ref{WWWWW})$ holds. 
\qed
\end{Proof}

\subsection{Proof of the main result}
\begin{Lemma}\label{last lemma}
Let $A\in \mathcal{B}^1$. Then $E(A)=0$ if and only if
 $\mu_{\delta_1}^{(k)}(A)=0$ for any $k\in\mathbb{N}$.
\end{Lemma}
\begin{Proof}\rm
Assume that $E(A)=0$. Then we see that $E^{(k)}(A)=0$ and 
\[
\mu_{\delta_1}^{(k)}(A)=(\delta_1,E^{(k)}(A)\delta_1)=0.
\]

Conversely, assume that $\mu_{\delta_1}^{(k)}(A)=0$ for any $k \in \mathbb{N}$. It is sufficient to prove that $E^{(k)}(A)=0$ for any $k \in \mathbb{N}$. Let $p$ ba a polynomial, then we see that 
\[
\mu_{p(H^{(k)})\delta_1}^{(k)}(A)=(p(H^{(k)})\delta_1,E^{(k)}(A)p(H^{(k)})\delta_1)=\int_{A}|p(\lambda)|^2 \mu_{\delta_1}^{(k)}(d\lambda)=0.
\]
This implies that $E^{(k)}(A)p(H^{(k)})\delta_1=0$. 
Since $\delta_1$ is a cyclic vector for $H^{(k)}:l^2(\mathbb{N}) \rightarrow l^2(\mathbb{N})$, 
$\{p(H^{(k)})\delta_1 \in l^2(\mathbb{N})\mid \text{ $p$ is a polynomial}\}$ is dense in $l^2(\mathbb{N})$. Therefore $E^{(k)}(A)=0$ follows. 
\qed
\end{Proof}
\begin{Lemma}\label{last last lamma}
Let $A\in \mathcal{B}^1$ and $A \subset (0,4)$. Then $\tilde{E}(A)=0$ if and only if
 $\mu_{\psi}^{(k)}(A)=0$ for any $k\in\mathbb{N}$ and any $\psi=f(H^{(k)})\delta_1$, with the first kind $f$.
Moreover, ${\rm dim}_{*}\tilde{E}={\rm dim}_{*}\mu_{\psi}^{(k)}$ and ${\rm dim}^{*}\tilde{E}={\rm dim}^{*}\mu_{\psi}^{(k)}$ follow for any $k\in \mathbb{N}$ and any $\psi=f(H^{(k)})\delta_1$, with the first kind $f$.
\end{Lemma}
\begin{Proof}\rm
Assume that $\tilde{E}(A)=0$. Then we see that $E^{(k)}(A)=0$ and 
\[
\mu_{\psi}^{(k)}(A)=(\psi,E^{(k)}(A)\psi)=0.
\]

Assume that $\mu_{\psi}^{(k)}(A)=0$ for any $k\in\mathbb{N}$ and $\psi=f(H^{(k)})\delta_1$, where $f$ is the first kind. 
Let $f_n \in C^{\infty}_0(\frac{1}{n}, 4-\frac{1}{n})$, $|f_n|\leq 1$, and $f_n=1$ on the interval $(\frac{2}{n}, 4-\frac{2}{n})$, $n=1,2,...$. Let $\psi_n=f_n(H^{(k)})\delta_1$. 
Since $f_n$ is the first kind, $\mu_{\psi_n}^{(k)}(A)=0$ for any $k \in \mathbb{N}$.
 It is sufficient to prove that $E^{(k)}(A)=0$ for any $k \in \mathbb{N}$. 
 We see that 
\[
\mu_{\psi_n}^{(k)}(A)=(f_n(H^{(k)})\delta_1,E^{(k)}(A)f_n(H^{(k)})\delta_1)=\int_{A}|f_n(\lambda)|^2 \mu_{\delta_1}^{(k)}(d\lambda)=0.
\]
By the Lebeasgue's dominated convergence theorem, 
\[
0=\lim_{n\rightarrow \infty}\mu_{\psi_n}^{(k)}(A)=
\lim_{n\rightarrow \infty}\int_{A}|f_n(\lambda)|^2 \mu_{\delta_1}^{(k)}(d\lambda)
=\mu_{\delta_1}^{(k)}(A).
\]
By Lemma $\ref{last lemma}$, we see that $E^{(k)}(A)=0$. Then we prove the first part of our assertion. The second part is straightforward to prove by the first part  and the definition of the lower and upper Hausdorff dimensions.
\qed
\end{Proof}
\begin{proof_main}\rm
By Lemma $\ref{last last lamma}$, it is sufficient to prove that ${\rm dim}_{*}\mu_{\psi}^{(k)}={\rm dim}^{*}\mu_{\psi}^{(k)}=\Gamma$ for any $k\in \mathbb{N}$ and any $\psi=f(H^{(k)})\delta_1$, with the first kind $f$. By Lemma $\ref{former-result}$, we see that
\[
\Gamma \leq {\rm dim}_{*}\mu_{\psi}^{(k)}\leq{\rm dim}^{*}\mu_{\psi}^{(k)}.
\]
By Lemma $\ref{H}$ and Lemma $\ref{HH}$, for any $p>0$,
\[
{\rm dim}^*(\mu_{\psi}^{(k)}) \leq \beta_{\psi}^{(k)}(p)= \cfrac{p+1}{p+\frac{1}{\Gamma}}.
\]
This imlies that ${\rm dim}_{*}\mu_{\psi}^{(k)}={\rm dim}^{*}\mu_{\psi}^{(k)}=\Gamma$.
\qed
\end{proof_main}

\begin{appendices}

\section{}
\label{Decomposition of the graph Laplacian}
We discuss the decomposition of the graph Laplacian and represent the  graph Laplacian as a Jacobi matrix. See $[1]$.

We assume that $G=(V,E)$ is a spherically homogeneous tree. 
Let $\pi_n:l^2(S_n)\rightarrow l^2(S_{n+1})$, $n=0,1,... ,$ be defined by
\[
\pi_nf(u)= \sum_{v\in S_{n}:v \sim u} f(v), \:\: {\text \rm  u \in S_{n+1}}.
\]
The adjoint $\pi_n^*:l^2(S_{n+1})\rightarrow l^2(S_n)$ is given by
\begin{displaymath}
\pi_n^*g(u)=\sum_{v\in S_{n+1}: v\sim u}g(v),  \:\: {\text \rm u \in S_n}.
\end{displaymath}
\begin{Lemma}\label{3.1.1}
Let $f$, $g \in l^2(S_n)$. Then $(\pi_nf,\pi_ng)=g_n(f,g)$.
\end{Lemma}
\begin{Proof}\rm
Let $f$, $g \in l^2(S_n)$. Since $G$ is a spherically homogeneous, 
\[
\langle \pi_n f, \pi_n g \rangle=\sum_{u\in S_{n+1}}\overline{\pi_n f(u)} \pi_n g(u)
=g_n\sum_{u\in S_{n}}\overline{ f(u)}  g(u).
\]
\qed
\end{Proof}

We see that $V$ is a disjoint union $V=\cup_{n=0}^{\infty}S_n$, 
and that $l^2(V)=\bigoplus_{n=0}^{\infty}l^2(S_n)$. 
Let $\Pi:l^2(V)\rightarrow l^2(V)$ be defined by 
$\Pi = \bigoplus_{n=0}^{\infty}\pi_n$.

\begin{Lemma}\label{RT}
Let $f \in \mathcal{D}$. Then $Af=(\Pi +\Pi^*)f$.
\end{Lemma}
\begin{Proof}\rm
Let $f \in \mathcal{D}$ and $u \in S_n$. 
Since $G$ is a spherically homogeneous tree, 
$u$ is adjacent with only vertices in $S_{n-1}$ and $S_{n+1}$. Therefore we see that
\[
(\Pi +\Pi^*)f(u)= \sum_{v\in S_{n-1};v\sim u}f(v)+ \sum_{v\in S_{n+1};v\sim u}f(v)
=\sum_{v \in V;v\sim u}f(v)
=Af(u).
\]
\qed
\end{Proof}
Let $\alpha_n= \#S_n={\rm dim}(l^2(S_n))$, $n=0,1,...$. 
Suppose that $\{ e^{(n)}_k\}_{k=1}^{\alpha_n}$ is a CONS  of $l^2(S_n)$. 
Then we can construct a CONS $\{ e^{(n+1)}_k\}_{k=1}^{\alpha_{n+1}}$ of $l^2(S_{n+1})$ by the following procedure.
Let 
$
e^{(n+1)}_k = \| \pi_n e^{(n)}_k\|^{-1}\pi_n e^{(n)}_k
$, $k=1,2,...,\alpha_n$.
By Lemma $\ref{3.1.1}$, 
$\{e^{(n+1)}_k\}_{k=1}^{\alpha_{n}}$ is an ONS of $l^2(S_{n+1})$. 
If $\alpha_n=\alpha_{n+1}$, then $\{e^{(n+1)}_k\}_{k=1}^{\alpha_{n}}$ is a CONS of $l^2(S_{n+1})$.  
If $\alpha_n<\alpha_{n+1}$,
by the Gram-Schmidt process, we can obtain $e^{(n+1)}_k\in l^2(S_{n+1})$, $k=\alpha_n+1,...,\alpha_{n+1}$, such that 
$\{e^{(n+1)}_k\}_{k=1}^{\alpha_{n}}\cup\{e^{(n+1)}_k\}_{k=\alpha_n+1}^{\alpha_{n+1}}$
 is a CONS of $l^2(S_{n+1})$.

Suppose that a CONS of $l^2(S_0)$ is given. 
Then we can costruct a CONS $\{e^{(n)}_k\}_{k=1}^{\alpha_n}$ of $l^2(S_n)$, 
$n=0,1,... ,$ inductively.
Hence, 
$
%\begin{displaystyle}
\bigcup_{n=0}^{\infty}\{e^{(n)}_k\}_{k=1}^{\alpha_n}
%\end{displaystyle}
$ is a CONS of $l^2(V)$.
%の完全正規直交系になる. 

Assume that
$
%\begin{displaystyle}
\sup_{n=0,1,...} \alpha_n=\infty
%\end{displaystyle}
$, and let  $\alpha_{-1}=0$.
Since   
$\{\alpha_n\}_{n=0}^{\infty}$ is non-decreasing, 
there exists a unique $N(k)\in \mathbb{N}\cup \{0\}$ such that 
$\alpha_{N(k)-1}< k \leq \alpha_{N(k)}$ for every $k \in \mathbb{N}$.
We see that 
\begin{displaymath}
\bigcup_{n=0}^{\infty}\{e^{(n)}_k \mid k=1,2,...,\alpha_n \}=
\bigcup_{k=1}^{\infty}\{e^{(n)}_k \mid n=N(k),N(k)+1,...  \}.
\end{displaymath}

\begin{Lemma}\label{invariance of M_k}
Let the closed subspace $M_k$ of $l^2(V)$, $k=1,2,...$, be defined by
\begin{displaymath}
M_k = \overline{\langle
\{e^{(n)}_k \mid n=N(k),N(k)+1,... 
\}\rangle}.
\end{displaymath}
Then $M_k$ is invariant under $A$, $D$ and $H$.
\end{Lemma}
\begin{Proof}\rm
By the definition of $e^{(n)}_k$ and 
Lemma \ref{3.1.1}, we see that 
\begin{eqnarray}
\Pi e^{(n)}_k&=&\| \pi_n e^{(n)}_k\| e^{(n+1)}_k ,
\nonumber
\\
\Pi^* e^{(n)}_k&=&
\begin{cases}
\cfrac{g_{n-1}}{\| \pi_{n-1} e^{(n-1)}_k\|}\; e^{(n-1)}_k  & (n \geq N(k)+1),
\\
\bm{o}&(n=N(k)).
\nonumber
\end{cases}
\end{eqnarray} 
This implies that $M_k$ is invariant under $\Pi$ and $\Pi^*$, and hence, 
by Lemma \ref{RT}, we see that 
$M_k$ is invariant under $A$.
Since $G$ is a spherically homogeneous tree, we have
\begin{equation}
D e^{n}_k =
\begin{cases}
(g_n+1) e^{(n)}_k & (n\geq 1),\\
g_0 e^{(0)}_1 & (n=0).
\end{cases}
\nonumber
\label{3.3}
\end{equation}
Hence, $M_k$ is also inavariant under $D$. 
Since $H=\overline{D-A}$, we see that $M_k$ is invariant under $H$.
\qed
\end{Proof}
By Lemma \ref{invariance of M_k},
let $H^{(k)},A^{(k)},D^{(k)}:M_k\rightarrow M_k$, $k=1,2,...$, be defined by the restriction of $H,A$ and $D$  to $M_k$, respectively.
We see that $H^{(k)}$ is self-adjoint and 
$H=\bigoplus_{k=1}^{\infty}H^{(k)}$. 

We consider the matrix representation of $H^{(k)}$ 
with respect to the CONS $\{e^{(n)}_k \mid n=N(k),N(k)+1,...  \}$ of $M_k$ for $k=2,3,...$. Then it follows for $n,m \geq N(k)$ that
\begin{eqnarray}
(e_k^{(n)},H^{(k)}e_k^{(n)})&=&(e_k^{(n)},D^{(k)}e_k^{(n)})=g_n+1,
\nonumber
\\
(e_k^{(n)},H^{(k)}e_k^{(n+1)})&=&
-(e_k^{(n)},A^{(k)}e_k^{(n+1)})=-\sqrt{g_n},
\nonumber
\\
(e_k^{(n)},H^{(k)}e_k^{(m)})&=&0, \text{ if $|n-m|\geq2$ } .
\nonumber
\end{eqnarray}
We have the matrix representation of $H^{(1)}$. 
It follows for $n,m \geq N(1)=0$ that
\begin{eqnarray}
(e_1^{(n)},H^{(1)}e_1^{(n)})&=&
(e_1^{(n)},D^{(1)}e_1^{(n)})=
\begin{cases}
g_0&(n=0),\\
g_{n}+1&(n \geq 1),
\end{cases}
\nonumber
\\
(e_1^{(n)},H^{(1)}e_1^{(n+1)})&=&
-(e_1^{(n)},A^{(1)}e_1^{(n+1)})=-\sqrt{g_n},
\nonumber
\\
(e_1^{(n)},H^{(1)}e_1^{(m)})&=&0,  \text{ if $|n-m|\geq 2$}.
\nonumber
\end{eqnarray}
Let $k,n\in \mathbb{N}$ and let 
$d_{k}=(d_k(n))_{n=1}^{\infty}$ and $a_{k}=(a_{k}(n))_{n=1}^{\infty}$ be defined by 
\begin{eqnarray}
d_k(n)&=&
(e_{k}^{(n+N(k)-1)},D^{(k)}e_{k}^{(n+N(k)-1)}),\nonumber\\
a_{k}(n)&=&
(e_{k}^{(n+N(k))},A^{(k)}e_{k}^{(n+N(k)-1)}).\nonumber
\end{eqnarray}
We can identify $H^{(k)}:l^2(\mathbb{N})\rightarrow l^2(\mathbb{N})$, $k=1,2,...$, with the following Jacobi matrix :
\begin{displaymath}
H^{(k)}=
\renewcommand{\arraystretch}{1.6}
\left(
\begin{array}{ccccccc}
d_{k}({1})&-a_{k}({1})\\
-a_{k}({1})&d_{k}({2})&-a_{k}({2})\\
&-a_{k}({2})&d_{k}({3})&-a_{k}({3})\\
&&-a_{k}({3})&\ddots&\ddots\\
&&&\ddots\\
\end{array}\right).
\renewcommand{\arraystretch}{1}
\end{displaymath}
This implies our assertion of Lemma $\ref{jacobi identification}$. 
Similary, we can identify $A^{(k)}, D^{(k)}:l^2(\mathbb{N}) \rightarrow l^2(\mathbb{N})$ with the following Jacobi matrices :
\begin{displaymath}
A^{(k)}=
\renewcommand{\arraystretch}{1.6}
\left(
\begin{array}{ccccccc}
0&a_{k}({1})\\
a_{k}({1})&0&a_{k}({2})\\
&a_{k}({2})&0&a_{k}({3})\\
&&a_{k}({3})&\ddots&\ddots\\
&&&\ddots\\
\end{array}\right)
\renewcommand{\arraystretch}{1},
D^{(k)}=
\renewcommand{\arraystretch}{1.6}
\left(
\begin{array}{ccccccc}
d_{k}({1})&\\
&d_{k}({2})&\\
&&d_{k}({3})&\\
&&&\ddots&\\
&&&\\
\end{array}\right).
\renewcommand{\arraystretch}{1}
\end{displaymath}

\section{}
In this section, we introduce a result about the Fourier analysis of the fractal measure. 
Let $B_r(x)=[x-r,x+r]\subset \mathbb{R}$. Let 
$\mathcal{L}$ be the Lebeage measure on $\mathbb{R}$, and 
$\mu:\mathcal{B}^1 \rightarrow [0,\infty]$ be a locally finite measure.
Let $M_{\mu}f:\mathbb{R} \rightarrow \mathbb{R}$ be defined by
\[
M_{\mu}f(x)
=
\sup_{r>0}
\cfrac{1}{\mu(B_r(x))}
\int_{B_r(x)}|f|d\mu
\]
for $f \in L^1_{\rm loc}(\mathbb{R}, d\mu)$, 
where we take $\frac{0}{0}=0$ if $\mu(B_r(x))=0$. 
$M_{\mu}f$ is measurable and called the Maximal function. 

\begin{comment}
\begin{Lemma}\label{Besicovitch}
There exists $C_0>0$ such that for any bounded subset $A\subset \mathbb{R}$, and any closed covering $\{B_{r_x}(x)\subset \mathbb{R} \mid x \in A\}$ of $A$ , 
there exists a countable subset $\{x_j\}_{j=1}^{\infty}$ of $A$ such that   
$
A\subset \bigcup_{j=1}^{\infty}B_j
$
and for any $y \in \mathbb{R}$,
\[
\mathbbm{1}_A(y)
\leq
\sum_{j=1}^{\infty}
\mathbbm{1}_{B_j}(y)
\leq
C_0,
\]
where $B_j=B_{r_{x_j}}(x_j)$.
\end{Lemma}
\begin{Proof}
\rm
\cite[Section II, 18, Theorem18.1]{real analysis}
\qed
\end{Proof}
\end{comment}

\begin{Lemma}\label{max}
$M_{\mu}:L^p(\mathbb{R},d\mu)\rightarrow L^p(\mathbb{R},d\mu)$ is bounded for any $p \in (1, \infty)$.
\begin{comment}
It follows for any $f \in L^1(\mathbb{R}, d\mu)$ and $s>0$, 
\[
\mu(\{
x \in \mathbb{R} \mid M_{\mu}f(x)>s
\})
\leq C_0s^{-1}\|f\|_{L^1}.
\]
Especially $M_{\mu}:L^1(\mathbb{R},d\mu)\rightarrow L^1(\mathbb{R},d\mu)$ is 
weak $(1,1)$ type quasi-linear. Moreover, 
\end{comment}
\end{Lemma}
\begin{Proof}\rm
Let $E_s^{n}=\{x \in \mathbb{R} \mid |x|\leq n, M_{\mu}f(x)>s\}$ and 
$x \in E_s^n$. There exists $r_x>0$ such that 
\[
\int_{B_{r_x}(x)}|f|d\mu\geq s\mu(B_{r_x}(x)).
\]
Note that $\{B_{r_x}(x)\mid x \in E_s^n \}$ is a Besicovitch covering of $E_s^n$. 
By \cite[II, 18 The Besicovitch covering theorem, Theorem18.1]{real analysis}, we see that there exists a countable subcollections $\{B_j^n\}_{j=1}^{\infty}$ of $\{B_{r_x}(x)\mid x \in E_s^n \}$ such that $\{B_j^n\}_{j=1}^{\infty}$ is a closed covering of $E_s^n$ 
and there exist $C>0$ which is independent of $E_s^n$ such that for any $x \in \mathbb{R}$, 
\[
\1_{E_s^n}(x) \leq \sum_{j=1}^{\infty} \1_{B_j^n}(x) \leq C. 
\]
Hence we have 
\[
s\;\mu(E_s^n)\leq s\sum_{j=1}^{\infty}\mu(B_j^n)
\leq 
\sum_{j=1}^{\infty} \int_{B_j^n}|f|d\mu \leq C\int_{\mathbb{R}}|f|d\mu.
\]
Let $n \rightarrow \infty$. Then we see that for any $f \in L^1(\mathbb{R}, d\mu)$ and $s>0$, 
\[
\mu(\{
x \in \mathbb{R} \mid M_{\mu}f(x)>s
\})
\leq Cs^{-1}\|f\|_{L^1}.
\]
This implies that $M_{\mu}:L^{1}(\mathbb{R},d\mu)\rightarrow L^{1}(\mathbb{R},d\mu)$ is weak $(1,1)$ type.
We also see that $M_{\mu}:L^{\infty}(\mathbb{R},d\mu)\rightarrow L^{\infty}(\mathbb{R},d\mu)$ is
weak $(\infty,\infty)$ type. 
Thus we have our assertion by \cite[VIII, 9 The Marcinkiewicz interpolation theorem, Theorem 9.1]{real analysis}.
\qed
\end{Proof}

We consider the Fourier transformation of the fractal measure. 
Let $f \in L^1(\mathbb{R},d\mu)$ and 
let $\widehat{f\mu}(\xi)$, $\xi \in \mathbb{R}$,  be defined by
\[
\widehat{f\mu}(\xi)=\int_{\mathbb{R}}f(x)e^{-i\xi x} \mu(dx).
\]
\begin{Lemma}
Suppose $\mu$ be a finite measure. Then 
\[
\int_{\mathbb{R}}
|\widehat{f\mu}(\xi)|^2 e^{-t\xi^2}d\xi<\infty.
\]
for any $f\in L^2(\mathbb{R}, d\mu)$ and $t>0$. 
\end{Lemma}
\begin{Proof}
\rm
Let $f\in L^1(\mathbb{R},d\mu)$. We see that
\begin{eqnarray}
\int_{\mathbb{R}}|\widehat{f\mu}(\xi)|e^{-t\xi^2}d\xi
&\leq&
\int_{\mathbb{R}}\int_{\mathbb{R}}|f(x)|\mu(dx)e^{-t\xi^2}d\xi
<\infty.
\nonumber
\end{eqnarray}
This implies that 
$L^1(\mathbb{R}, d\mu)\ni f \rightarrow \widehat{f\mu}\in L^1(\mathbb{R}, e^{-t\xi^2}d\xi)$ is bounded. 
Since $\mu$ is a finite measure, 
$L^{\infty}(\mathbb{R}, d\mu)\ni f \rightarrow \widehat{f\mu}\in L^{\infty}(\mathbb{R}, e^{-t\xi^2}d\xi)$ is bounded. We have our assertion by the Riesz interpolation theorem.
\qed
\end{Proof}

\begin{Definition}Let $\alpha \in (0,1)$. 
We say that a measure $\mu$ is uniformly $\alpha$-H\"{o}lder continuous, 
if there exists 
$\widetilde{C}_1>0$ such that $\mu(I)<\widetilde{C}_1\mathcal{L}(I)^{\alpha}$ 
for any interval $I\subset \mathbb{R}$ with $\mathcal{L}(I)<1$.
\end{Definition}

\begin{Lemma}\label{fourierdecay}
Let $\mu$ be a uniformly $\alpha$-H\"{o}lder continuous and finite measure.
Then there exists $\widetilde{C}_2=\widetilde{C}_2(\alpha,\mu)>0$ such that for any $f \in L^2(\mathbb{R}, d\mu)$,
\begin{equation}
\sup_{0<t \leq1}
t^{\frac{1-\alpha}{2}}
\int_{\mathbb{R}}
|\widehat{f\mu}(\xi)|^2 e^{-t\xi^2}d\xi<\widetilde{C}_2\|f\|_{L^2}^2.
\nonumber
\end{equation}
\end{Lemma}
\begin{Proof}\rm
Let $f \in L^2(\mathbb{R}, d\mu)$. We see that
\begin{eqnarray}
t^{\frac{1-\alpha}{2}}
\int_{\mathbb{R}}
|\widehat{f\mu}(\xi)|^2 e^{-t\xi^2}d\xi
&=&
t^{\frac{1-\alpha}{2}}
\int_{\mathbb{R}}\mu(dx)
\int_{\mathbb{R}}\mu(dy)
f(x)\overline{f(y)}
\int_{\mathbb{R}}e^{-t\xi^2-i\xi(x-y)}d\xi 
\nonumber\\
&=&
\pi
t^{-\frac{\alpha}{2}}
\int_{\mathbb{R}}\mu(dx)
\int_{\mathbb{R}}\mu(dy)
f(x)\overline{f(y)}
e^{-\frac{(x-y)^2}{4}}
\nonumber\\
&=&
\pi
t^{-\frac{\alpha}{2}}
\int_{\mathbb{R}}\mu(dx)
\int_{\mathbb{R}}\mu(dy)
f(x)\overline{f(y)}
\int_{|x-y|}^{\infty}
\cfrac{r}{2t}\;e^{-\frac{r^2}{4t}}dr
\nonumber\\
&\leq&
\pi
t^{-\frac{\alpha}{2}}
\int_{\mathbb{R}}\mu(dx)\;
|f(x)|
\int_{0}^{\infty}
dr
\cfrac{r}{2t}\;
e^{-\frac{r^2}{4t}}
\int_{B_r(x)}\mu(dy)
|f(y)|
\nonumber\\
&\leq&
\pi
t^{-\frac{\alpha}{2}}
\int_{\mathbb{R}}\mu(dx)\;
|f(x)|
\int_{0}^{\infty}
dr
\cfrac{r}{2t}\;
e^{-\frac{r^2}{4t}}
\mu(B_r(x))
M_{\mu}f(x).
\nonumber
\end{eqnarray}
Since $\mu$ is uniformly $\alpha$-H\"{o}lder continuous and finite, 
there exists $\widetilde{C}_2^{\prime}=\widetilde{C}_2^{\prime}(\alpha,\mu)>0$ such that for any $t \in (0,1]$,
\begin{eqnarray}
t^{-\frac{\alpha}{2}}
\int_{0}^{\infty}
dr
\cfrac{r}{t}\;
e^{-\frac{r^2}{4t}}
\mu(B_r(x))
&\leq&
\mu({\mathbb{R}})
t^{-\frac{\alpha}{2}}
\int_{1}^{\infty}
dr
\cfrac{r}{t}\;
e^{-\frac{r^2}{4t}}
+\widetilde{C}_1t^{-\frac{\alpha}{2}}
\int_{0}^{1}
dr
\cfrac{r^{1+\alpha}}{t}\;
e^{-\frac{r^2}{4t}}
\nonumber\\
&\leq&
2\mu({\mathbb{R}})
t^{-\frac{\alpha}{2}}
e^{-\frac{1}{4t}}
+2^{2+\alpha}\widetilde{C}_1
\int_{0}^{\infty}
s^{1+\alpha}
e^{-s^2}
ds
\nonumber
\\
&\leq&
\widetilde{C}_2^{\prime}.
\nonumber 
\end{eqnarray}
Let $\widetilde{C}_2=\pi \widetilde{C}_2^{\prime}>0$. By Schwarz inequality, we see that for any $t\in (0,1)$,
\begin{eqnarray}
&&
t^{\frac{1-\alpha}{2}}
\int_{\mathbb{R}}
|\widehat{f\mu}(\xi)|^2 e^{-t\xi^2}d\xi
\leq
\widetilde{C}_2
\int_{\mathbb{R}}\mu(dx)\;
|f(x)|
M_{\mu}f(x)
\leq
D_2
\|f\|_{L^2}\|M_{\mu}f\|_{L^2}.
\nonumber
\end{eqnarray}
By Lemma \ref{max}, we have our assertion.
\qed
\end{Proof}
\begin{Lemma}\label{cor}
Let $\mu$ be a uniformly $\alpha$-H\"{o}lder continuous and finite measure.
Then there exists $\widetilde{C}_3=\widetilde{C}_3(\alpha,\mu)>0$ such tha for any $f \in L^2(\mathbb{R}, d\mu)$,
\[
\sup_{T\geq 1}
T^{\alpha-1}
\int_{0}^T
|\widehat{f\mu}(\xi)|^2d\xi
\leq \widetilde{C}_3\|f\|^2.
\]
\end{Lemma}
\begin{Proof}\rm
Let $t \in (0,1)$ and $T=t^{-\frac{1}{2}}$. 
By Lemma \ref{fourierdecay}, we see that for any $T>1$,
\begin{eqnarray}
\widetilde{C}_2\|f\|^2
\geq
T^{\alpha-1}
\int_{\mathbb{R}}
|\widehat{f\mu}(\xi)|^2 e^{-(\frac{\xi}{T})^2}d\xi
\geq
e^{-1}
T^{\alpha-1}
\int_{0}^T
|\widehat{f\mu}(\xi)|^2 d\xi.
\nonumber
\end{eqnarray}
This implies our assertion.
\qed
\end{Proof}

\begin{Lemma}\label{cor1}
Let $\mu$ be a uniformly $\alpha$-H\"{o}lder continuous and finite measure.
Then there exists $\widetilde{C}_4=\widetilde{C}_4(\alpha,\mu)>0$ such that for any $f \in L^2(\mathbb{R}, d\mu)$,
\begin{equation}
\sup_{T\geq 1}
T^{\alpha-1}
\int_{0}^{\infty}
e^{-\frac{t}{T}}
|\widehat{f\mu}(t)|^2dt
\leq
\widetilde{C}_4\|f\|^2.
\nonumber
\end{equation}
\end{Lemma}
\begin{Proof}\rm 
Let $f \in L^2(\mathbb{R}, d\mu)$.  
Then, by Lemma \ref{cor}, we see that for any $T>1$, 
\begin{eqnarray}
T^{\alpha-1}
\int_{0}^{\infty}e^{-\frac{t}{T}}|\widehat{f\mu}(t)|^2dt
&=&
\lim_{N\rightarrow \infty}
T^{\alpha-1}
\int_{0}^{T(N+1)}e^{-\frac{t}{T}}|\widehat{f\mu}(t)|^2dt
\nonumber
\\
&=&
\lim_{N\rightarrow \infty}
T^{\alpha-1}
\sum_{n=0}^{N}
\int_{Tn}^{T(n+1)}e^{-\frac{t}{T}}|\widehat{f\mu}(t)|^2dt
\nonumber
\\
&=&
\sum_{n=0}^{\infty}
(n+1)^{1-\alpha}e^{-n}
\{T(n+1)\}^{\alpha-1}
\int_{Tn}^{T(n+1)}|\widehat{f\mu}(t)|^2dt
\nonumber
\\
&\leq&
\widetilde{C}_3 \|f\|^2
\sum_{n=0}^{\infty}
(n+1)^{1-\alpha}e^{-n}.
\nonumber
\end{eqnarray}
This implies our assertion.\qed
\end{Proof}

\section{}
In this section we give some lemmas about quadratic form theory which is used in Seciotin 4.
Let $\mathcal{H}$ ba a complex Hilbert space. Let $\mathfrak{s}:\mathcal{H}\times \mathcal{H}\rightarrow \mathbb{C}$
be a closed sesquilinear form, and 
$T:\mathcal{H}\rightarrow \mathcal{H}$ be a closed linear operator.
We say that $\mathfrak{s}$ is symmetric, if 
$\mathfrak{s}(f,g)=\overline{\mathfrak{s}(g,f)}$
for $f,g \in \mathcal{D}(\mathfrak{s})$, and that
$\mathfrak{s}$ is sectorial, if 
there exist $r\in \mathbb{R}$ and $\theta \in (-\frac{\pi}{2}, \frac{\pi}{2})$ such that 
for $f \in \mathcal{D}(\mathfrak{s})$ with $\|f\|=1$,
\[
{\rm arg}(\mathfrak{s}[f]-r)\leq \theta,
\]
where $\mathfrak{s}[f]= \mathfrak{s}(f,f)$. We say that $T$ is sectorial, if 
there exist $r\in \mathbb{R}$ and $\theta \in (-\frac{\pi}{2}, \frac{\pi}{2})$ such that 
for $f \in \mathcal{D}(\mathfrak{s})$ with $\|f\|=1$,
%$\mathfrak{s}[f]= \mathfrak{s}(f,f)$として
\[
{\rm arg}((f,Tf)-r)\leq \theta,
\]
and that $T$ is m-accretive, if 
${\rm Re}(f,Tf)\geq 0$ for $f \in \mathcal{D}(T)$ and 
$(T+\lambda)^{-1}$ is bounded and $\|(T+\lambda)^{-1}\|\leq ({\rm Re}\lambda)^{-1}$
for $\lambda \in \mathbb{C}$ with ${\rm Re} \lambda > 0$.
In particular, $T$ is said to be quasi m-accretive, if 
there exists $\gamma\in \mathbb{R}$ such that 
$T+\gamma$ is m-accretive, and 
$T$ is said to be m-sectorial, if 
$T$ s quasi m-accretive and sectorial.
\begin{Lemma}
Let $\mathfrak{s}:\mathcal{H}\times \mathcal{H}\rightarrow \mathbb{C}$ 
be a densely defined, closed, and sectorial sesquilinear form. Then 
there exist a unique m-sectorial operator $S:\mathcal{H}\rightarrow \mathcal{H}$
such that for $f \in \mathcal{D}(\mathfrak{s}), g\in \mathcal{D}(S)$,
\begin{equation}
\mathfrak{s}(f,g)=(f,Sg).
\nonumber
\end{equation}
\end{Lemma}
\begin{Proof}\rm
\cite[VI, \S2, Theorem 2.1]{Kato}
\end{Proof}
\begin{Lemma}\label{quadratic thm}
Let $\mathfrak{t}:\mathcal{H}\times \mathcal{H}\rightarrow \mathbb{C}$ be 
a densely defined, closed, and symmetric form bounded from below, and 
let $T$ be the self-adjoint operator associated with $\mathfrak{t}$.
Suppose that $\mathfrak{s}:\mathcal{H}\times \mathcal{H}\rightarrow \mathbb{C}$ is 
a relatively bounded sesquilinear form with respect to $\mathfrak{t}$ such that 
for any $f \in \mathcal{D}(\mathfrak{t}) \subset \mathcal{D}(\mathfrak{s})$,
\[
|\mathfrak{s}[f]|\leq a\mathfrak{t}[f]+b\|f\|^2, \qquad 0<a<1,\;b\geq0.
\]
Then $\mathfrak{t}^{\prime}= \mathfrak{s}+\mathfrak{t}$ is sectorial and closed. 
Let $T^{\prime}$ be the m-sectorial operators associated with $\mathfrak{t}^{\prime}$.
If $0<\gamma<1$, $z\in \rho(T)$ and 
\[
2\|(aT+b)(T-z)^{-1}\|\leq \gamma <1,
\]
then
$z\in \rho (T^{\prime})$ and 
\[
\|(T^{\prime}-z)^{-1}-(T-z)^{-1}\|\leq \cfrac{4\gamma}{(1-\gamma)^2}\|(T-z)^{-1}\|.
\]
\end{Lemma}
\begin{Proof}\rm
\cite[VI, \S3, Theorem 3.9]{Kato}
\qed
\end{Proof}

\end{appendices}

\section*{Acknowledgement}
This work was supported by JST SPRING, Grant Number JPMJSP2136.

\end{document}